\documentclass[11pt]{amsart}
\usepackage{amsmath,amssymb,amsthm}
\usepackage[latin1]{inputenc}
\usepackage{mathrsfs}
\usepackage{version, tabularx, multicol}
\usepackage{graphicx,float,psfrag}
\usepackage{subfigure}
\usepackage{epstopdf}
\usepackage{color,latexsym,amsfonts}
\usepackage{mathtools}

\newcommand{\reff}[1]{(\ref{#1})}

\theoremstyle{plain}
\newtheorem{theo}{Theorem}[section]
\newtheorem{theo*}{Theorem}
\newtheorem{cor}[theo]{Corollary}
\newtheorem{prop}[theo]{Proposition}
\newtheorem{lem}[theo]{Lemma}
\newtheorem{defi}[theo]{Definition}
\theoremstyle{remark}
\newtheorem{rem}[theo]{Remark}

\newcommand{\ca}{{\mathcal A}}

\newcommand{\cf}{{\mathcal F}}

\newcommand{\cn}{{\mathcal N}}

\newcommand{\ct}{{\mathcal T}}

\newcommand{\cz}{{\mathcal Z}}

\newcommand{\E}{{\mathbb E}}

\newcommand{\F}{{\mathbb F}}

\newcommand{\N}{{\mathbb N}}
\renewcommand{\P}{{\mathbb P}}

\newcommand{\R}{{\mathbb R}}

\newcommand{\T}{{\mathbb T}}

\newcommand{\ra}{{\rm a}}
\newcommand{\rbb}{{\rm b}}
\newcommand{\rd}{{\rm d}}

\newcommand{\rN}{{\rm N}}
\newcommand{\rP}{{\rm P}}
\newcommand{\rE}{{\rm E}}

\newcommand{\bP}{{\bar \P}}
\newcommand{\bE}{{\bar \E}}

\newcommand{\bt}{{\mathbf t}}
\newcommand{\bff}{{\mathbf f}}

\newcommand{\bu}{{\bar u}}

\newcommand{\ind}{{\bf 1}}

\newcommand{\Card}{{\rm Card}\;}

\newcommand{\inv}[1]{\mathop{\frac{1}{ #1}}\nolimits}
\newcommand{\expp}[1]{\mathop {\mathrm{e}^{ #1}}}

\renewcommand{\root}{\partial}

\newcommand{\lb}{[\![}
\newcommand{\rb}{]\!]}

\def\beqlb{\begin{eqnarray}}\def\eeqlb{\end{eqnarray}}
 \def\beqnn{\begin{eqnarray*}}\def\eeqnn{\end{eqnarray*}}
\def\rar{\rightarrow}

\def\ez{\varepsilon}
\def\lz{\lambda}

\title[Stationary CBI]{Some properties of stationary continuous state branching processes}

\date{\today}

\author{Romain Abraham}
\address{Romain Abraham,
Institut Denis Poisson,
Universit\'{e} d'Orl\'{e}ans,
Universit\'e de Tours,
CNRS,
France}
\email{romain.abraham@univ-orleans.fr}

\author{Jean-Fran\c{c}ois Delmas}
\address{Jean-Fran\c{c}ois Delmas,
 CERMICS, Ecole des Ponts, France}
\email{delmas@cermics.enpc.fr}

\author{Hui He}
\address{Hui He,
 School of Mathematical Sciences, Beijing Normal University, P.R. China}
\email{hehui@bnu.edu.cn}

\begin{document}



\thanks{ We thank Beijing  Normal University, Orl\'eans University and
  Ecole  des Ponts  for  the  invitations during  which  this paper  was
  written.   Hui He  also  thanks  S. Feng  for  enlightening
  discussions. This research was partially supported by NSFC
  (No. 11671041 and 11531001)}

\begin{abstract}
  We consider  the genealogical  tree of  a stationary  continuous state
  branching  process  with  immigration.    For  a  sub-critical  stable
  branching mechanism, we  consider the genealogical tree  of the extant
  population at  some fixed time and  prove that, up to  a deterministic
  time-change,  it is  distributed  as  a continuous-time  Galton-Watson
  process with  immigration. We  obtain similar  results for  a critical
  stable branching mechanism when only looking at immigrants arriving in
  some  fixed  time-interval.   For  a  general  sub-critical  branching
  mechanism, we consider the number of individuals that give descendants
  in  the  extant  population.   The associated  processes  (forward  or
  backward in time)  are pure-death or pure-birth  Markov processes, for
  which we compute the transition rates.
\end{abstract}

\maketitle

\section{Introduction}

\subsection{State of the art}

Inference of the genealogical tree of some given population (or of a sample of extant individuals) is a central question in evolutionary biology (see for instance \cite{HS11}) and, to perform this task by the usual maximum likelihood method, the distribution of this genealogical tree must be known.

The most popular model in this  context is the Wright-Fisher model where
the genealogical tree of a sample  of extant individuals is given by the
Kingman coalescent  \cite{K75}. One  major feature of  this model  is to
consider a constant  size population although many  extensions have been
proposed   to   take   into   account  population   size   change   (see
e.g.  \cite{GT94}).  Other models  have  also  been considered  where  the
distribution  of  the genealogical  tree  or  a  sample of  the  current
population  can be  explicitely  described:  linear birth-death  process
\cite{NMH94}, continuous time Galton-Watson trees \cite{HJR19,J19}, Brownian tree \cite{AP05}  see also \cite{AD18}, splitting
trees \cite{La10}. Some recent results on the coalescent process
associated with some branching process by time-reversal can be found in
\cite{WLY19, JL19, FMM19}.
\medskip

We consider here continuous state branching processes with immigration so that the total population size is stationary. More precisely, let $\psi$ be a sub-critical branching mechanism of the form
\begin{equation}
\label{eq:def-psi}
\psi(\lambda)=\alpha\lambda +\beta\lambda^2+\int_{(0,+\infty)}\left(\expp{-\lambda r}-1+\lambda r\right)\pi(dr),
\end{equation}
where   $\alpha=\psi'(0)>0$  (which   implies  that   $\psi$  is   sub-critical),
$\beta\ge 0$  and $\pi$  is a  $\sigma$-finite measure  on $(0,+\infty)$
such   that   $\int_{(0,+\infty)}  (r\wedge   r^2)\pi(dr)<+\infty$ and
which furthermore satisfies:
\begin{equation}\label{eq:grey-kappa}
\int^{+\infty}\frac{d\lambda}{\psi(\lambda)}<+\infty \quad\text{(Grey
  condition) and}\quad \int_{0+}\left(\frac{1}{\lambda \alpha} -
\frac{1}{\psi(\lambda)}\right)\,  d\lambda<\infty. 
\end{equation}
The  Grey condition implies in particular that $\beta>0$ or
$\int_{(0,1)}r\pi(dr)=+\infty$.

A  continuous  state branching  process  (CB  process  for short)  is  a
positive real valued Markov process $(Y_t,t\ge 0)$ that satisfies the following
branching  property:  the  process   $Y$  starting  from  $Y_0=x+x'$  is
distributed  as  $Y^{(1)}+Y^{(2)}$  where $Y^{(1)}$  and  $Y^{(2)}$  are
independent copies  of $Y$ starting respectively  from $Y_0^{(1)}=x$ and
$Y_0^{(2)}=x'$. The  distribution of  the process  $Y$ is  then uniquely
determined by  its branching  mechanism, see  Section \ref{sec:def-psi}.
As we only consider sub-critical branching mechanisms together with Grey
condition \eqref{eq:grey-kappa},  the  population  becomes  a.s.
extinct in finite time. We denote by $c(t)$ the probability of
non-extinction at time $t>0$ under the canonical measure which is
defined by:
\[
\int_{c(t)} ^{+\infty } \frac{d\lambda}{\psi(\lambda)}=t.
\]
The second condition in \reff{eq:grey-kappa} insures that the following
limit is well defined:
\begin{equation}
   \label{eq:def-kappa}
 \kappa=\lim_{t\to+\infty}c(t)\expp{\alpha t}\in
(0, +\infty )
\end{equation}
where according to Lemma 1 in \cite{L03}, $\kappa$ satisfies 
$c^{-1}(\kappa)=\int_0^{\kappa}\left(\frac{1}{\alpha\lambda}-\frac{1}{\psi(\lambda)}\right)d\lambda
$.
\medskip

  One way to avoid this extinction  is to add an
immigration  characterized  by a  function  $\phi$ defined on $\R^+
$  which describes  the
intensity of the  immigration and the size of  the immigrant population,
see for example \cite{L20} and references therein. A natural immigration
function, which  appears for instance  when conditioning the  initial CB
process on non-extinction, see \cite{L07,CD12}, is given by: for $\lambda\geq 0$,
\begin{equation}
   \label{eq:def-phi}
\phi(\lambda)=\psi'(\lambda) -
\alpha=2\beta\lambda+\int_{(0,+\infty)}\left(1-\expp{-\lambda  
    r}\right)r\pi(dr).
\end{equation}
We can  then consider  a CB  process with  immigration (CBI process for
short) indexed  by $\R$,
$Y=(Y_t,\ t\in \R)$, whose one-dimensional distributions are constant in
time.   Some  properties  of  this process  have  been  investigated  in
\cite{CD12}. By convention, the stationary case will correspond to a
sub-critical branching mechanism $\psi$ and the corresponding
immigration $\phi$ given by \reff{eq:def-phi}. 
We shall denote by $\bu$ the Laplace transform of $Y_t$, see
\reff{Zlaplace01},  which is given
by:
\begin{equation}
   \label{eq:def-bu0}
\bu (\lambda)=\frac{\kappa \alpha \expp{-\alpha
    c^{-1}(\lambda)}}{\psi(\lambda)}\cdot
\end{equation}

The description  of the genealogy of  CB processes is done  using L\'evy
trees (see \cite{DLG02}), and of CBI processes  as a real
tree with an  infinite spine on which some L\'evy  trees are grafted. As
the population size in our  CBI processes  is stationary,
we can look at the extant population at any fixed time, say $t=0$ in all
the paper. We want to describe the distribution of the genealogical tree
of this extant  population. A complete description of  this genealogy is
already  done  in  \cite{AD18}   for  a  quadratic  branching  mechanism
$\psi(\lambda)=\alpha\lambda   +\beta\lambda^2$    together   with   the
description  of  the  genealogical  tree  of  a  sample  of  the  extant
population. We focus in this paper on general branching mechanisms.

\subsection{Main results}
For a  general sub-critical branching mechanism  $\psi$, the description
of the genealogy  of the extant population, in the  stationary case, can
be  seen as  a  birth process  (forward  in time)  and  a death  process
(backward in time) coming from infinity.  Let $1+M_t^0$ be the number of
descendants  of   the  extant  population   forward  in  time   at  time
$t\in        (-\infty,        0)$.        Notice        that        a.s.
$\lim_{t\rightarrow   -\infty  }   M_t^0=0$.    The  ancestral   process
$( M^0_t,t<  0)$ describes  in some  sense the  genealogy of  the extant
population at  time 0.  Asymptotics of  $M_t^0$ as $t$ increases  to $0$
are  given  in \cite{CD12}  (see  also  references therein  for  related
results on  coalescent processes).   We have  the following  result, see
Propositions \ref{prop:q-d} and \ref{prop:q-b}.

\begin{theo}
   \label{theo:main-b-d}
   Assume  $\psi$  given  by  \reff{eq:def-psi}  is   sub-critical
   (\textit{i.e.} $\alpha>0$)  and
   satisfies conditions \reff{eq:grey-kappa} and $\phi$ is given by \reff{eq:def-phi}.  

\begin{itemize}
   \item[(i)]
The forward in time
   process $(M^0_{t},\,  t<0)$ is a c\`ad-l\`ag  inhomogeneous pure birth Markov
   process  starting from $0$ at time
   $-\infty $ with birth rate given by
    for $m>n\geq 0$ and
   $t>0$ :
\[
q^\rbb_{n,m}(-t)
= \frac{(m+1)}{(m+1-n)!} \,
c(t)^{m-n}\, \left|\psi^{(m-n+1)}\bigl(c(t)\bigr) \right| .
\]

\item[(ii)]  The backward in time process $(M^0_{(-t)-},\, t>0)$ is a
   c\`ad-l\`ag inhomogeneous pure death Markov process    starting from
   $+\infty $ at time $0$, with death rate given by
    for $n>m\geq 0$ and
   $t>0$: 
\[
q^\rd_{n,m}(t)
=\binom{n+1}{m}\, \frac{\left|\bar{u}^{(m)}\bigl(c(t)\bigr)\right|}
{ \left|\bar{u}^{(n)}\bigl(c(t)\bigr)\right|}\, 
\left| \psi^{(n-m+1)}\bigl(c(t)\bigr) \right|.
\]
\end{itemize}
\end{theo}

\medskip
We now consider a stable branching mechanism:
\begin{equation}
   \label{eq:def-psi-b}
\psi(\lambda)=\alpha\lambda +\gamma \lambda^b
\end{equation}
with $\alpha>0$, $\gamma>0$ and  $b\in(1,2]$. The case $b=2$ corresponds
to $\pi=0$ in \reff{eq:def-psi}, and the case $b\in (1, 2)$ corresponds
to $\beta=0$  and $\pi(dr)$ equal  (up to a multiplicative  constant) to
$r^{-b-1}  \, dr$.  See  Remark \ref{rem:birth-rate-b}  for an  explicit
computation  of   the  birth   rate  for   $1<b<2$.  See   also  Remark
\ref{rem:death-rate-2}   for  an  explicit
computation of  the birth and death  rates in the quadratic  case $b=2$,
which already appears in Proposition 3.2 and 3.3 in \cite{BD16}.
\medskip 

We now  present a deterministic time  change for which the  genealogy of
the  extant  population   (forward  in  time)  in the stationary case  is   a  time  homogeneous
Galton-Watson process with  immigration.  The time change  relies on the
extinction probability  $c(t)$ of  the associated  CB process  under the
canonical measure  which is given (see Example 3.1 p.~62 in
\cite{l:mvbmp} where  $\bar v_t$ 
corresponds to $c(t)$ in our setting) for $t>0$ by:
\begin{equation}
   \label{eq:def-c}
c(t)=\left(\frac{\alpha}{\gamma\left(\expp{(b-1)\alpha
        t}-1\right)}\right)^{\frac{1}{b-1}}. 
\end{equation}
We consider the  time change   $T(t)=-R^{-1}(t)$ where:
\begin{equation}
   \label{eq:def-R}
R(t)=\log \left(\frac{\tilde \psi(c(t))}{\tilde \psi(0)}\right)
\quad\text{with}\quad 
\tilde \psi(\lambda)=\frac{\psi(\lambda)}{\lambda}=\alpha + \gamma
\lambda^{b-1}
\end{equation}
and we consider the process $\tilde M=(\tilde M_t= M^0_{T(t)},t>0)$. 
The main result of the paper is the following
theorem.

\begin{theo}\label{thm:main}
 Assume  $\psi$   is given  by  \reff{eq:def-psi-b} (with $\alpha>0$ and
 $b\in (1,2]$) and $\phi$  by \reff{eq:def-phi}. The time-changed
 ancestral process $\tilde M$ is distributed as a 
continuous-time Galton-Watson process with immigration.
\end{theo}
The characteristics of the Galton-Watson process (length of the
branches, immigration rate, offspring distribution, immigration size)
are precised in Theorem \ref{thm:gwi}. This process may also be viewed as a sized-biased continuous-time Galton-Watson process, see Remark \ref{rem:size-biased}.

\medskip  As a  corollary of  this theorem,  we study  the sizes  of the
families of the extant population  ranked according to their immigration
time. The vector  of the sizes of  these families in the  stable case is
distributed as the  jumps of a time-changed subordinator  which yields a
Poisson-Kingman  distribution (see Remark \ref{RemSub}). In the quadratic case (see Corollary
\ref{propPD}) this corresponds  to a Poisson-Dirichlet distribution. The computations  of Proposition \ref{prop:moments} prove
that for $b\in (1, 2)$,   the  distribution  of  the  sizes of  these
families  is  not  a  Poisson-Dirichlet
distribution  since  a sized-biased sample  of the  vector of
sizes is not Beta-distributed (except maybe for one very particular case).
\medskip

In   the   stable   critical  case   $\psi(\lambda)=   \lambda^b$   with
$b\in (1, 2]$, the previous results  do not make sense since the total
population  size  is  always  infinite  and  the  ancestral  process  is
trivially infinite at all times.  To  get a finite extant population, we
restrict our attention to the extant individuals whose initial immigrant
arrived  after some  fixed  time $-T$.   Theorem \ref{thm:main}  remains
valid  in this  setting with  a different  change of  time, see  Theorem
\ref{thm:gwiT}.  
\medskip

If the Grey condition is not  satisfied, it is always possible to define
the  genealogy of  a  CBI process  whose immigration mechanism   is given  by $\phi$  in
\reff{eq:def-phi}, see Corollary 3.3 in  \cite{CD12}. However, the ancestral
process is again trivially infinite at all time. This is for example the
case       for        the       Neveu's        branching       mechanism
$\psi(\lambda)=\lambda \log(\lambda)$ which appears as the natural limit
of  the stable  branching mechanism  $\psi(\lambda)=\lambda^b$ when  $b$
goes down  to 1.  There  is a natural link  between the CB  with Neveu's
branching   mechanism  and   the  Bolthausen-Sznitman   coalescent,  see
\cite{BLe00}. Inspired  by this result, the following result, see
Proposition  \ref{prop:BS},  gives
that looking backward  the genealogical tree
in the stationary stable case, one recovers, as $b$ decreases to 1, 
the Bolthausen-Sznitman   coalescent. 
Let $T>0$ and $n\geq 1$.  Conditionally  on the  number of ancestors at time $-T$ of
the extant population being $n$, that is on 
$\{M_{-T}^0=n-1\}$, we label them from $1$ to $n$
uniformly at random.  Define a continuous time 
process $(\Pi^{T,  [n]}(t), t\geq T)$ taking values in the partitions of
 $[n]=\{1,\,2,\,\cdots, n\}$, by   $\Pi^{T,  [n]}(t)$ is the
partition of $[n]$  such that $i$ and  $j$ are in the same  block if and
only if  the $i$-th and $j$-th  individuals at level $-T$  have the same
ancestor at level $-t$. 

\begin{theo}
\label{prop:BS-intro} 
 Assume  $\psi$   is given  by  \reff{eq:def-psi-b} (with $\alpha>0$ and
 $b\in (1,2]$) and $\phi$  by \reff{eq:def-phi}.  The law of $(\Pi^{T, [n]}(T\expp{\gamma
  t}), t\geq0)$ conditionally on $\{M_{-T}^0=n-1\}$,  converges in the
sense of finite dimensional distribution to a Bolthausen-Sznitman
coalescent as $b$ decreases to 1. 
\end{theo}

\subsection{Organisation of the paper}
Sections \ref{sec:notations} and  \ref{sec:CBI} are respectively devoted
to recall  known results on  CB process  and their genealogy  using real
trees, and  on CBI process  and the definition of  the number of
descendants of the  extant population $M=( M^0_t,t< 0)$.  For the stable
sub-critical  setting, we  present  in  Section \ref{sec:GW-result}  the
proof of Theorem \ref{thm:gwi} (and  thus of Theorem \ref{thm:main}) and
the study of  the sizes of the families of  the extant population ranked
according to  their immigration time.  We  compute the birth and  death rates of the  process $M$ in
Section \ref{sec:bd-rate} and apply these expressions in Sub-section  \ref{sec:BS-coal} to prove the convergence of the  ancestral process as $b$ goes down to
1 towards  the Bolthausen-Sznitman coalescent.  We  provide some results
in  Section   \ref{sec:cstable}  for   the  critical   stable  branching
mechanism.

\section{Notations}
\label{sec:notations}

Concerning probability measures and expectations, we shall use $\rP$ and
$\rE$ for usual real random variables or processes, $\P$ and $\E$ for
Lévy trees or Lévy forest,  and  $\bP$ and $\bE$ for the corresponding
stationary cases which involve immigration. 

The process $Y$ usually refer to a CB or CBI process (under $\rP$) and $Z$ usually
refer to  a CB (under $\P$) or a CBI (under $\bP$) built on a Lévy tree
or a Lévy forest.  
\medskip

We write  $\N=\{0, 1, \ldots\}$ for the set of integers and $\N^*=\{1,
2, \ldots\}$.  for the set of positive integers.

\subsection{Continuous branching
  processes}\label{sec:def-psi}

We refer to \cite{Bi76,Gr74,L20} for a presentation and general results on
CB   processes. We
recall that a CB  process with branching mechanism $\psi$ (denoted CB($\psi)$) is a c\`ad-l\`ag non-negative real-valued Markov process $Y=(Y_t,\ t\ge 0)$ whose transition kernels are characterized, for every $s,t,\lambda\ge 0$, by
\begin{equation}
   \label{eq:Laplace-cb}
\rE\left[\expp{-\lambda Y_{s+t}}\Bigm| Y_s\right]=\expp{-u(t,\lambda)\, Y_s},
\end{equation}
where $(u(\lambda, t); t\geq 0, \lambda\geq 0)$ is the unique
non-negative solution of the integral  equation
\begin{equation}
   \label{eq:dif-u-t}
 u(\lambda, t)+\int_0^ t \psi\bigl(u(\lambda,
s)\bigr)\, ds =\lambda, 
\end{equation}
or equivalently the unique non-negative solution of the integral equation
\begin{equation}\label{eq:def-u}
\int_{u(\lambda,t)}^\lambda \frac{dr}{\psi(r)}=t.
\end{equation}
We  set for $t>0$:
\begin{equation}
\label{eq:def_c}
c(t)=u(+\infty, t)=\lim_{\lambda\to+\infty}u(\lambda,t)
\end{equation}
which is finite thanks to the Grey condition, see
\eqref{eq:grey-kappa}.
\medskip

We denote by  $\rN$ the  canonical measure of the CB process $Y$:
\textit{i.e.} if $(Y^i)_{i\in I}$ are the atoms of a Poisson point
measure with intensity $r \rN(dY)$, then the process $(\tilde Y_t,t\ge 0)$
defined by 
\[
 \tilde Y_t=\sum_{i\in I}Y^i_t
\]
is distributed as $Y$ conditionally on $Y_0=r$.  In particular, we have for $\lambda,t\ge 0$:
\[
\rN\left[1-\expp{-\lambda
    Y_t}\right]=\lim_{r\to0}\frac{1}{r}\rE\left[1-\expp{-\lambda
    Y_t}\Bigm| Y_0=r\right]=u(\lambda,t)
\]
and the function $c(t)$ satisfies for $t>0$:
\begin{equation*}
c(t)=\rN[Y_t>0],\quad u(c(t), s)=c(t+s)\quad\text{and}\quad c'(t)=-\psi(c(t)).
\end{equation*}

\subsection{Continuous branching  process with immigration}
In general, the immigration mechanism  $\phi$ is the Laplace exponent of
a subordinator.   A stationary  CBI   process  associated  with  the  branching
mechanism $\psi$ and  the immigration mechanism $\phi$  is a c\`ad-l\`ag
non-negative   real-valued  Markov   process  $Y=(Y_t,t\in   \R)$  whose
transition   kernels  are   characterized,   for   every  $s,t\in   \R$,
$\lambda\geq 0$, by
\[
\rE\left[\expp{-\lambda Y_{s+t}}\Bigm| Y_s\right]=\exp\left(-u(\lambda,t)Y_s-\int_0^t\phi\bigl(u(\lambda,r)\bigr)dr\right)
\]
where $u$ is still the function given by \eqref{eq:def-u}. We refer to
\cite{KW71} for more results on CBI processes. 

Under the  Grey condition for  the branching mechanism $\psi$,  when the
immigration  mechanism  $\phi$ is  given  by  \reff{eq:def-phi}, then  the
process $Y$  can be viewed  as the  CB process with  branching mechanism
$\psi$ conditioned  on non-extinction.   This observation  motivates the
particular choice for this immigration mechanism.  \medskip

Assume that $\psi$ defined by \reff{eq:def-psi} satisfies \reff{eq:grey-kappa}
and that $\phi$ is given by  \reff{eq:def-phi}. Recall the function $c$
defined by \reff{eq:def_c}. 
Recall  $\bu$ defined in  \reff{eq:def-bu0}. 
Then, according to
Corollary  3.13   in  \cite{CD12},  we have that  
 for  every
$\lambda\ge0$ and  $t\in\R$,
\begin{eqnarray}\label{Zlaplace01}
\rE \left[\expp{-\lambda Y_t}\right]=\bu (\lambda).
\end{eqnarray}

\subsection{Real trees and L\'evy trees}

We refer to   \cite{DMT96,Ev08} for general results on real trees and to 
 \cite{DLG05} for Lévy trees. We recall that a metric space $(\bt,d)$ is a real tree if the following two properties hold for every $u,v\in\bt$.
\begin{itemize}
\item[(i)] There is a unique isometric map $f_{u,v}$ from $[0,d(u,u)]$ into $\bt$ such that
\[
f_{u,v}(0)=u\qquad\mbox{and}\qquad f_{u,v}\bigl(d(u,v)\bigr)=v.
\]
\item[(ii)] If $\varphi$ is a continuous injective map from $[0,1]$ into $\bt$ such that $\varphi(0)=u$ and $\varphi(1)=v$, then the range of $\varphi$ is also the range of $f_{u,v}$.
\end{itemize}
The range of the map $f_{u,v}$ is denoted $\lb u,v\rb$. It is the unique
continuous path that links $u$ to $v$ in the tree.  In order to simplify
the notations,  we often omit the  distance $d$ in the  notation and say
that $\bt$ is a real tree.

A rooted real tree is a  real tree $(\bt,d)$ with a distinguished vertex
$\root$  called the  root.  Two  real  trees (resp.  rooted real  trees)
$\bt_1$  and $\bt_2$  are  called  equivalent if  there  is an  isometry
(resp. a root-preserving isometry) that maps $\bt_1$ onto $\bt_2$.
We set $\T$ the set of all equivalence classes of rooted compact real
trees. We endow the set $\T$  with the pointed Gromov-Hausdorff distance
(see \cite{Ev08}) and the associated Borel $\sigma$-field. The set $\T$
is then Polish. 
\medskip

Let $\bt\in\T$ be a rooted tree. We define a partial order $\prec$
(called the genealogical order) on $\bt$ by: 
\[
u\prec v\iff u\in\lb \root ,v\rb \setminus\{v\}
\]
and we say in this case that $u$ is an ancestor of $v$. The height of a vertex
$u\in\bt$ is defined by
\[
H(u)=d(\root,u),
\]
and we denote by $H(\bt)=\sup\{d(\root,u),\ u\in\bt\}$ the height of the
tree $\bt$. Let  $a>0$. The truncation of $\bt$ at level $a$ is the tree
$ {\rm{Tr}}_a(\bt)=\{u\in\bt,\ H(u)\le a\}$, 
and the population of the tree $\bt$ at level $a$ is the sub-set
\begin{equation}
   \label{eq:def-cz}
\cz_\bt(a)=\{u\in\bt,\ H(u)=a\}.
\end{equation}
We denote by $(\bt^{(i),*},i\in I)$ the connected components of the open
set $\bt\setminus {\rm{Tr}}_a(\bt)$. For every $i\in I$, there exists a
unique point $\root _i\in z_\bt(a)$ such that  $\root _i\in\lb \root,
u\rb$ for every $u\in\bt^{(i),*}$. We then set
$\bt^{(i)}=\bt^{(i),*}\cup\{\root _i\}$ so that $\bt^{(i)}$ is a compact
rooted real tree with root $\root _i$ and we consider the point measure
on $\cz_\bt (a)\times \T$: 
\[
\cn_a^\bt=\sum_{i\in I} \delta_{(\root _i,\bt^{(i)})}.
\]

We now recall the definition of  the excursion measure associated with a
$\psi$-L\'evy  tree  from  \cite{DLG05}.   Let  $\psi$  be  a  branching
mechanism  defined by  \reff{eq:def-psi}. Then,  there exists  a measure
$\N$ on $\T$ such that:
\begin{itemize}
\item[(i)] \textbf{Existence of a local time.}
For every $a\ge 0$ and for $\N(d\ct)$-a.e. $\ct\in\T$, there exists a
finite measure $\ell^a$ on $\ct$ such that 
\begin{itemize}
\item[(a)] $\ell^0=0$ and, for every $a>0$, $\ell^a$ is supported on $\cz_\ct(a)$.
\item[(b)] For every $a>0$, $\{\ell^a\ne 0\}=\{H(\ct)>a\}$, $\N(d\ct)$-a.e.
\item[(c)] For every $a>0$, we have $\N(d\ct)$-a.e. for every bounded continuous function $\varphi$ on $\ct$,
\begin{align*}
\langle \ell^a,\varphi\rangle & =\lim_{\varepsilon\to0+}\frac{1}{c(\varepsilon)}\int\cn_a^\ct(du\ d\ct')\varphi(u)\ind_{\{H(\ct')\ge\varepsilon\}}\\
& =\lim_{\varepsilon\to0+}\frac{1}{c(\varepsilon)}\int\cn_{a-\varepsilon}^\ct(du\ d\ct')\varphi(u)\ind_{\{H(\ct')\ge\varepsilon\}}.
\end{align*}
\end{itemize}
\item[(ii)] \textbf{Branching property.}
For every $a>0$, the conditional distribution of the point measure $\cn_a^\ct(du\ d\ct')$, under the probability measure $\N(d\ct\,|\, H(\ct)>a)$ and given $\rm{Tr}_a(\ct)$, is that of a Poisson point measure on $\cz_\ct(a)\times \T$ with intensity $\ell^a(du)\N(d\ct')$.
\item[(iii)] \textbf{Regularity of the local time process.}
We can choose a modification of the process $(\ell^a,a\ge 0)$ in such a way that the mapping $a\longmapsto \ell^a$ is $\N(d\ct)$-a.e. c\`ad-l\`ag for the weak topology on finite measures on $\ct$.
\item[(iv)] \textbf{Link with CB processes.}
Under $\N(d\ct)$, the process $(\langle \ell^a,1\rangle,a\ge 0)$ is distributed as a CB($\psi$) process under $\rN$.
\end{itemize}
\medskip

If  necessary, we  shall write  $\ell^a(\ct)$ for  $\ell^a$ in  order to
stress the dependence  in the Lévy tree $\ct$. We  define the population
size process as  $Z=(Z_a, a\geq 0)$, where the ``size''  of the population
at level $a$ is given by:
\begin{equation}\label{eq:def-Z}
Z_a=\langle \ell^a(\ct),1\rangle.
\end{equation}
We recall that under $\N$, the process $Z$ is distributed as $Y$ under
the canonical measure $\rN$. 

\subsection{Forests}

\begin{defi}[Forest and leveled forest]
A forest is   a  family $\bff=(\bt_i)_{i\in I}$,  at most countable,  of
elements of $\T$. 
A leveled forest is a  family $\bar\bff=(h_i,\bt_i)_{i\in I}$, at most
countable,  of elements of $\R\times\T$. 
We denote by $\F$ (resp. $ \bar\F$) the set of (resp. leveled) forests. 
\end{defi}

If $\bar\bff=(h_i,\bt_i)_{i\in I}$ is a leveled forest, denoting by $d_i$ the distance in the tree $\bt_i$ and $\root_i$ the root of $\bt_i$, we can associate with it a tree $(\bt(\bar\bff),\bar d)$ by
\[
\bt(\bar\bff)=\R\sqcup \left(\bigsqcup _{i\in I}\bt_i^*\right)
\]
where $\sqcup$ denotes the disjoint union of sets, $\bt_i^*=\bt_i\setminus\{\root_i\}$, and, for every $u,v\in \bt(\bar\bff)$,
\[
\bar d(u,v)=\begin{cases}
|u-v| & \mbox{if }u,v\in \R,\\
d_i(u,v) & \mbox{if } u,v\in\bt_i^*,\\
|u-h_i|+d_i(\root_i,v) & \mbox{if }u\in\R\mbox{ and }v\in \bt_i^*,\\
d_i(\root_i,u)+|h_i-h_j|+d_j(\root_j,v) & \mbox{if } u\in\bt_i^*,\ v\in\bt_j^* \mbox{ with }i\ne j.
\end{cases}
\]

\begin{rem}
It is easy to check that $\bt(\bar\bff)$ is indeed a real tree. It is neither rooted nor compact, and can be seen as a tree with a two-sided infinite spine (the set $\R$).
\end{rem}

\begin{rem}
If $\bff=(h_i,\bt_i)_{i\in I}$ and $(h_i,\tilde\bt_i)_{i\in I}$ are two families of real numbers and real trees such that, for every $i\in I$, the trees $\bt_i$ and $\tilde \bt_i$ are equivalent, then the trees constructed by the above procedure are also equivalent, so the construction is valid for families of elements of $\R\times\T$.
\end{rem}

We extend the notion of ancestor in the tree $\bt(\bar\bff)$ by
\[
u\prec v\iff \begin{cases}
u<v & \mbox{if }u,v\in\R,\\
u\le h_i & \mbox{if }u\in\R\mbox{ and }v\in\bt_i^*,\\
u\prec_i v & \mbox{if }u,v\in\bt_i^*,
\end{cases}
\]
where $\prec_i$ denotes the genealogical order in the tree $\bt_i$. We also extend the notion of height of a vertex $u\in\bt(\bff)$ by
\[
H(u)=\begin{cases}
u & \mbox{if }u\in\R,\\
h_i+H_i(u) & \mbox{if }u\in\bt_i^*,
\end{cases}
\]
where $H_i$ denotes the height of a vertex in the tree $\bt_i$.

\begin{defi}(Ancestral  tree)\label{def:ancestral} Let  $\bar\bff$ be  a
  leveled forest and  let $\bar  \bt=\bt(\bar\bff)$ be  its associated
  tree.  For  every  $a\in\R$,  we   define  $\cz_{\bar  \bt}  (a)$  the
  population at height  $a$, by \reff{eq:def-cz} with  $\bt$ replaced by
  $\bar  \bt$  and  the  ancestral   tree  $\ca_{\bar  \bt}(a)$  of  the
  population at level $a$ by
\[
\ca_{\bar \bt}(a)=\cz_{\bar \bt}(a)\cup {\rm Anc}(\cz_{\bar \bt}(a)),
\]
where ${\rm Anc}(\cz_{\bar \bt}(a))=\cup_{v\in \cz_{\bar \bt}(a)} \{u\in
\bar \bt, \, u \prec v\} $ is the set of all the ancestors in
$\bar \bt$ of the vertices of $\cz_{\bar \bt}(a)$.
\end{defi}
When there is no confusion we write $\ca$  for $\ca_{\bar \bt}$. 
\section{The stationary Lévy tree}
\label{sec:CBI}

\subsection{Random forests, CB  and CBI processes}
\label{sec:random-forest}
Let  $\psi$  be  a  branching
mechanism  defined by  \reff{eq:def-psi}.
For $r>0$, we denote by $\P_r(d\bff)$ the probability distribution on
$\F$ of the random forest $\cf=(\ct_i)_{i\in I}$ given by the atoms of a poisson point measure on $\T$ with intensity $r\N(d\bt)$.
Under $\P_r$, the family $(\ell^a(\ct_i))_{i\in I}$ of the corresponding
local times at level $a\ge 0$ is well defined, and we define the local
time at level $a$ of the forest $\cf$ by 
\begin{equation}\label{eq:def-ell-forest}
\ell^a(\cf)=\sum_{i\in I} \ell^a(\ct_i).
\end{equation}
Let the  size-population process  $Z=(Z_a, a\geq  0)$ be  defined (under
$\P_r$) by \reff{eq:def-Z} with the local time $\ell^a(\ct)$ replaced by
$\ell^a(\cf)$.  By property (iv) of the Lévy tree excursion measure, and
the definition of the probability measure $\P_r$, we get that under $\P_r$, the process $Z$ is a CB started
at time 0 from $r$.  \medskip

If $\bff=(\bt_i)_{i\in I}$ is a forest and $h\in\R$, the pair $(h,\bff)$
can be viewed as the leveled forest $(h,\bt_i)_{i\in I}$. Eventually, a
family of leveled forests $(\bar\bff_i)_{i\in I}$ can be viewed as a
leveled forest since a countable disjoint union of countable sets
remains countable. Conversely a tree  is a forest, thus the measure $\N(d\bt)$
on $\T$ can be viewed as a measure $\N(d\bff)$ on $\F$. 

We denote by $\bP(d\bar \bff)$ the probability distribution on
$\bar \F$ of the random leveled forest $\bar\cf=(h_i,\cf_i)_{i\in I}$
given by the atoms of a Poisson point measure on $\R\times\F$ with
intensity 
\[
\nu(dh, d\bff)= dh \left(\beta\N[d\bff]+\int_0^{+\infty}\pi(dr)\,
  \P_r(d\bff)\right), 
\]
and let $\bar \ct=\bt(\bar \cf)$ be the random tree associated with this
leveled forest. The random tree $\bar  \ct$ under $\bP$ can be viewed as
stationary  version of  the Lévy  tree with  branching mechanism  $\psi$
conditioned  on  non-extinction, see  \cite{CD12},
Section 3.  We call the random tree $\bar \ct$ the stationary Lévy tree.
\medskip

For every $i\in I$, the local time measure $\ell^a(\cf_i)$ at level $a$
of the leveled forest $(h_i,\cf_i)$ is a.s. well-defined by
\eqref{eq:def-ell-forest}. We then define, for every $a\in\R$, the local
time measure at level $a$ for the tree $\bar \ct$ by
\begin{equation}\label{eq:def-ell}
\ell^a(\bar \ct) =\sum_{i\in I}\ell^{a-h_i}(\cf_i)_i\ind_{\{h_i\le a\}}. 
\end{equation}
By standard  property of Poisson  point measures, we have  the following
result,  where   $Z=(Z_a,  a\in \R)$   is  defined  (under   $\bP$)  by
\reff{eq:def-Z}   with  the   local  time   $\ell^a(\ct)$  replaced   by
$\ell^a(\bar \ct)$.

\begin{prop}
\label{prop:Z=cbi}
  Under $\bP$,  the process $Z$  is a stationary CBI  process associated
  with  the branching  mechanism  $\psi$ and  the immigration  mechanism
  $\phi$ given by \reff{eq:def-phi}.
\end{prop}

\medskip

\subsection{Branching        points        of       the        ancestral
  tree}\label{sec:notations_death}
Recall $\bar \ct$ is defined under $\bar \P$ in the previous section. 
For $t\in \R$, we write $\ca(t)$ the ancestral tree $\ca_{\bar \ct} (t)$
of the population at level $t$ defined by Definition
\ref{def:ancestral}.  Notice that  $\bar\P $-a.s. $\ca(t)$  has only a finite number of
vertices at any level $s<t$ and we set for $s<t$:
\begin{equation}\label{eq:def-M}
M_s^t=\Card\{u\in\ca_{\bar \ct} (t),\ H(u)=s\}-1.
\end{equation}
The number  $M_s^t$ is  exactly the  number of  individuals of  the tree
$\bar \ct$ at level $s$ that have descendants at level $t$, the immortal
(or two-sided infinite) 
spine  being  excluded (which  explains  the  -1  in the  definition  of
$M_s^t$).

Under $\bar\P  $, since the  intensity $\nu(dh, d\bff)$ is  invariant by
translation  in $h$,  we get that the  distribution of  the ancestral  tree
$\ca(t)$ does not  depend on $t\in\R$.  Therefore, we can  fix the level
at which  the current population is  considered, say $t=0$, and  look at
the ancestral  process $M^0=(M_s^0,s<0)$  which is a  pure-birth process
starting at time $s=-\infty$ from 0.  \medskip

We define the jumping times of the process $M^0$  inductively by setting
\begin{eqnarray}\label{tau0}
\tau_0=\sup\{t>0,\ M_{-t}^0\ne 0\}
\end{eqnarray}
and for $n\ge 1$,
\begin{eqnarray}\label{taun}
\tau_n=\sup\{t<\tau_{n-1},\ M_{-t}^0\ne M_{(-t)-}^0\},
\end{eqnarray}
and we define the size of the $n$-th jump of the process $M^0$, $n\ge 0$, by
\begin{eqnarray}\label{jumpsize}
\xi_n=M_{-\tau_n}^0-M_{-(\tau_n)-}^0=M_{-\tau_n}^0-M_{-\tau_{n-1}}^0.
\end{eqnarray}

\begin{figure}[H]
\psfrag{T0}{$\tau_0$}
\psfrag{T1I}{$\tau_1^I$}
\psfrag{T1B}{$\tau_1^B$}
\includegraphics[width=10cm]{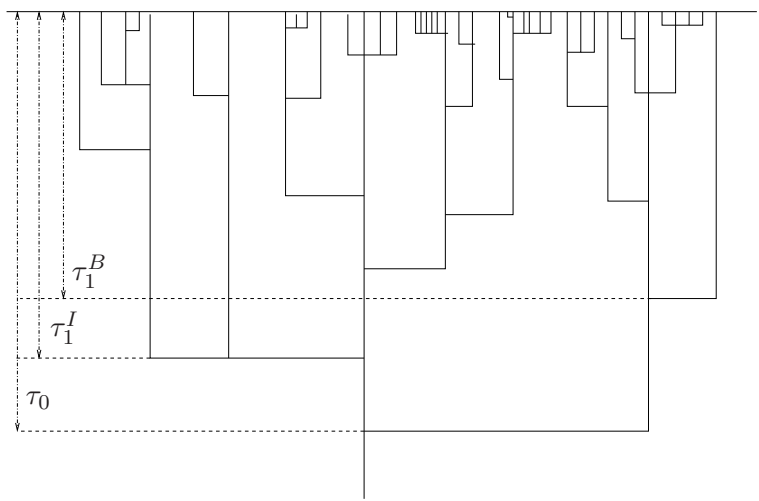}
\caption{The ancestral tree and the first jumping times.}
\end{figure}

In the sequel, we will distinguish between the jumps that are due to a new immigration (i.e. a branching point on the infinite spine) and those coming from a reproduction of an individual of the ancestral tree. For that purpose, recall that the tree is constructed from a random leveled forest $(h_i,\ct_i)_{i\in I}$.

We then define for every $n\ge 1$,
\begin{align}\label{tauIB}
\tau_n^I &=-\inf\{h_i>-\tau_{n-1},\ H(\ct_i)\ge -h_i\}\quad\mbox{and}\\
\tau_n^B &=-\inf\{t>-\tau_{n-1},\ M_t^0\ne M_{t-}^0\mbox{ and } t\ne h_i\ \forall i\in I\},
\end{align}
so that
\begin{equation}
   \label{eq:t=min}
\tau_n=\tau_n^I\vee \tau_n^B. 
 \end{equation} 
Thanks to  Theorem 2.7.1 of \cite{DLG02}, we have, for every
$r\in[0,1]$, every $t>u>0$ and every $n\in\N^*$,
\begin{equation}\label{eq:xi_1^B}
\bar\E [r^{\xi_1}\bigm|\tau_0=t, \xi_0=n, \tau_1^B=u, \tau_1^I<u]=g_t(t-u,r),
\end{equation}
where
\begin{equation}\label{eq:def-g}
g_t(s,r)=r\,
\frac{\psi'(c(t-s))-\gamma_\psi(c(t-s),(1-r)c(t-s))}
{\psi'(c(t-s))-\gamma_\psi(c(t-s),0)} ,
\end{equation}
with
\[
\forall a,b\ge 0,\qquad\gamma_\psi(a,b)=\begin{cases}
\frac{\psi(a)-\psi(b)}{a-b} & \mbox{if }a\ne b,\\
\psi'(a) & \mbox{if }a=b.
\end{cases}
\]
On the other hand, by standard properties of Poisson point measures (see
also Proposition 5.2 in \cite{CD12}), we have:
\begin{equation}\label{eq:xi_1^I}
\bar\E [r^{\xi_1}\bigm|\tau_0=t, \xi_0=n, \tau_1^I=u,
\tau_1^B<u]=1-\frac{\phi((1-r)c(u))}{\phi(c(u))}\cdot 
\end{equation}

\section{Properties of the ancestral process in the sub-critical  stable case}
\label{sec:GW-result}

In this section, the branching mechanism $\psi$, the  immigration
mechanism $\phi$, and the function $\tilde \psi$ are given by 
\reff{eq:def-psi-b},  \reff{eq:def-phi} and \reff{eq:def-R}, 
that is, for $\lambda\geq 0$:
\begin{equation}
   \label{eq:def-psi-b-s}
\psi(\lambda)=\alpha\lambda+\gamma \lambda^b,
\quad
\phi(\lambda)=b\gamma\lambda^{b-1},
\quad
\tilde \psi(\lambda)=\alpha + \gamma
\lambda^{b-1} , 
\end{equation}
with  $\alpha>0$, $\gamma>0$ and  $b\in (1,2]$.   We recall\footnote{
According to Example 3.1 p.~62 in
\cite{l:mvbmp} (where  $v_t(\lambda)$ 
corresponds to $u(\lambda,t)$ in our setting), we also have 
$u(\lambda, t)=\expp{-\alpha t} \lambda \left[1+ \gamma \, \alpha^{-1}\,
(1- \expp{-\alpha (b-1)t }) \lambda^{b-1}\right]^{-1/(b-1)}$.}
  the extinction  probability
$c(t)$  defined by  \eqref{eq:def_c}  and  given  by \reff{eq:def-c}, we
recall  and
explicit the Laplace transform of the CBI $\bu$ as well as the
constant $\kappa$ defined in \reff{eq:def-kappa} and \reff{eq:def-bu0}:
\begin{equation}
\label{Zlaplace02}
c(t)=\left(\frac{\alpha}{\gamma\left(\expp{(b-1)\alpha
        t}-1\right)}\right)^{\frac{1}{b-1}},
\quad
\bu(\lambda)=\left(1+\frac{\gamma}{\alpha}
  \lambda^{b-1}\right)^{-\frac{b}{b-1}}   
\quad\text{and}\quad
\kappa=\left(\frac{\alpha}{\gamma}\right)^{\frac{1}{b-1}}.
\end{equation}
The
expression of the function $g_t(s,r)$ of \eqref{eq:def-g} does not
depend on $s$ and $t$. We have for $r\in [0, 1]$:
\begin{equation}\label{eq:g_stable}
g_t(s,r)=g_B(r)
\quad\text{with}\quad
g_B(r)=\frac{br-1+(1-r)^b}{b-1}\cdot
\end{equation}
We also define the generating function $g_I$ by, for $r\in [0, 1]$:
\begin{equation}
   \label{eq:def-gi}
g_I(r)= \frac{(b-1)}{b}  g'_B(r)=1-(1-r)^{b-1}.
\end{equation}

\subsection{Distribution of the time-changed ancestral process}
We explicit the time change given in \reff{eq:def-R}: for $t>0$
\[
R(t)
=\log \left(\frac{\tilde \psi(c(t))}{\tilde \psi(0)}\right)
=\log\left(\frac{\expp{(b-1)\alpha
      t}}{\expp{(b-1)\alpha t}-1}\right). 
\]
The function $R$ is continuous and strictly decreasing; we also have
that   $\lim _{t\to 0}R(t)=+\infty$ and $\lim_{t\to+\infty}R(t)=0$. Thus
the function $R$ is one-to-one from $(0,+\infty)$ to $(0,+\infty)$.
We consider the time-changed ancestral process $\tilde M= (\tilde M_t,t\ge 0)$ defined by 
$\tilde M_0=0.$ and for $t>0$: 
\[
\tilde M_t =M_{T(t)}^0
\quad\text{with}\quad
T(t)=-R^{-1}(t). 
\]

The next theorem, whose proof is given in Section \ref{sec:pgwi},
is the  main result of this section. It states that the ancestral process
is a continuous-time Galton-Watson process with immigration (GWI process). 
\begin{theo}\label{thm:gwi}
  Consider     the    sub-critical     stable    branching     mechanism
  with immigration 
  \eqref{eq:def-psi-b-s}. The time-changed ancestral process $\tilde M$ is
  distributed under $\bP $ as a GWI process, $X=(X_t,t\ge 0)$, with:
\begin{itemize}
\item[(i)] $X_0=0$ a.s.;
\item[(ii)]  the branching rate of $X$ is 1;
\item[(iii)]  the offspring distribution has generating function $g_B$ defined
  in \reff{eq:g_stable};
\item [(iv)] the immigration rate is $\frac{b}{b-1}$;
\item[(v)]  the number of immigrants has generating function $g_I$
  defined in \reff{eq:def-gi}.
\end{itemize}
\end{theo}

Recall that the distribution of the process $X$ is characterized by the Markov property and its infinitesimal transition probabilities. Let us denote by $p=(p_n,n\ge 0)$ (resp. $q=(q_n,n\ge 0)$) the distribution on $\N$ associated with the generation function $g_B$ (resp. $g_I$).

First, since $p_0=g_B(0)=0$, we have, for every $t\ge 0$, $h>0$ and every $k<n$
\[
\rP(X_{t+h}=k|X_t=n)=0.
\]

Furthermore, by Equation \eqref{eq:def-gi}, we have, for every $n\ge 1$,
\[
np_n=\frac{b}{b-1}q_{n-1}.
\]
Therefore,
as $h\rightarrow0+$, we have for every $0\le n<k$,
\[
\rP(X_{t+h}=k|X_t=n)=\left(np_{k-n+1}+\frac{b}{b-1}q_{k-n}\right)h+o(h)=(k+1)p_{k-n+1}h+o(h).
\]
Eventually, since $p_1=g'_B(0)=0$, we have, for every $n\ge 0$, as $h\to 0+$,
\[
\rP(X_{t+h}=n|X_t=n)=1-\sum_{k=n+1}^{+\infty}(k+1)p_{k-n+1}h+o(h)=1-\left(\frac{b}{b-1}+n\right)h+o(h).
\]

To sum up, we have the following transition rates for the GWI process X as $h\to 0+$,
\begin{eqnarray}\label{rateX}
\rP(X_{t+h}=k|X_t=n)=
\begin{cases}
 (k+1)p_{k-n+1}h+o(h) & \text{ if $k\ge n+1$},\\
1-\left(\frac{b}{b-1}+n\right)h+o(h) & \text{ if $k=n$},\\
o(h) & \text{ otherwise}.
\end{cases}
\end{eqnarray}
In particular, if $(\tau'_n,n\ge 0)$ is the sequence of
jumping times of $X$ (with $\tau'_0=0)$, we have for $r\in [0, 1]$, $n,k\ge 0$,
\begin{equation}
   \label{eq:gf_jumps1}
\rE\left[r^{X_{\tau'_{n+1}}-X_{\tau'_n}}\bigm|X_{\tau'_n=k}\right]=g_{[k]}(r),
\end{equation}
where for $r\in [0, 1]$, 
\begin{align}
\nonumber
g_{[k]}(r)
&=
\frac{k}{k(b-1)+b}\left(br-1+(1-r)^b\right)
+\frac{b}{k(b-1)+b}\left(1-(1-r)^{b-1}\right) \\
\label{eq:gf_jumps2}
&= \frac{k(b-1)}{k(b-1)+b}g_B(r)+ \frac{b}{k(b-1)+b} g_I(r). 
\end{align}
\medskip

\begin{rem}\label{rem:size-biased}
  Let $\chi=(\chi_t, t\ge 0)$ be a continuous-time Galton-Watson process
  (GW  process)  with  branching   rate  1,  offspring  distribution
  $p$ and starting at $\chi_0=1$. Recall that the size-biased version of
  $\chi$ is the process $\hat \chi=(\hat \chi_t, t\ge 0)$ such that 
 for every
  $T>0$ and every bounded measurable functional $\varphi$, we have:
\begin{equation}
\label{eq:size-biased}
\rE\left[\varphi(\hat \chi_t,\, t\in [0, T])\right]
=\frac{1}{\rE[\chi_T]}\rE\left[\chi_T\, \varphi(\chi_t,\, t\in [0, T])\right].  
\end{equation}
Then, the  GWI process $X$ of Theorem \ref{thm:gwi}  is distributed as
$\hat \chi-1$. 

Indeed, the process $\hat \chi$ is a Markov process as a Doob
h-transform of a Markov process (the process $(\chi_t/\E[\chi_t],t\ge
0)$ is a martingale). Its transition rates are given by the following
computations. For every $t\ge 0$, $\varepsilon>0$ and every integers
$1\le n<k$, we have: 
\begin{align*}
\frac{1}{\varepsilon}\rP(\hat \chi_{t+\varepsilon}=k|\hat \chi_t=n) 
& =\frac{1}{\varepsilon}\frac{\rE[\ind_{\{\hat \chi
  _{t+\varepsilon}=k, \, \hat\chi_t=n\}}]}{\rE[\ind_{\{\hat \chi _t=n\}}]}\\
&  =\frac{1}{\varepsilon}\frac{\rE[\chi_{t+\varepsilon}\ind_{\{\chi_{t+\varepsilon}=k,
  \, \chi_t=n\}}]}{\rE[\chi_t\ind_{\{\chi_t=n\}}]} 
\frac{\rE[\chi_t]}{\rE[\chi_{t+\varepsilon}]}\\
&  =\frac{1}{\varepsilon}\frac{k}{n}
  \rP(\chi_{t+\varepsilon}=k|\chi_t=n)\frac{\rE[\chi_t]}
{\rE[\chi_{t+\varepsilon}]}\cdot  
\end{align*}
We deduce that for $0\le n<k$: 
\[
\lim_{\varepsilon \rightarrow 0+} 
\frac{1}{\varepsilon}\rP(\hat \chi_{t+\varepsilon}-1=k|\hat \chi_t-1=n) 
=\frac{k+1}{n+1}(n+1)p_{k-n+1}.
\]
According to the transition rates given in \eqref{rateX}, we deduce that
$X$ is distributed as $\hat \chi -1$.  
\end{rem}

The following result is an application of Theorem \ref{thm:gwi}. Recall
$\kappa$ defined in \reff{Zlaplace02}. 
\begin{cor}\label{cor:martlim} Let $X$ be the GWI process  defined in
  Theorem \ref{thm:gwi}. Then there exists a random variable $W$
  distributed as $\kappa Z_0$
  under $\bP$, such
  that 
\begin{eqnarray}\label{martlim}
\lim_{t\rar\infty}\expp{-\frac{t}{b-1}}X_t\overset{a.s.}{=}W.
\end{eqnarray}
\end{cor}
%=:hZ(0).$

\begin{proof}
It is known from Corollary 6.5 in \cite{CD12} that
 a.s. $\lim_{s\downarrow0}\frac{M_{-s}}{c(s)}=Z_0$.
Using the expressions of $R$ and $c$, we have:
\begin{equation}
   \label{eq:CR-1}
c\bigl(R^{-1}(t)\bigr)
=\left(\frac{\alpha}{\gamma}(\expp{t}-1)\right)^{\frac{1}{b-1}}, 
\end{equation}
and thus $\lim_{t\rar\infty}
\expp{-\frac{t}{b-1}}c(R^{-1}(t))=\kappa$. Then
\reff{martlim} follows readily from Theorem \ref{thm:gwi}.
\end{proof}
\begin{rem}
  If  a GW process or a GWI process has  finite  offspring mean and
  finite immigration mean,
  then limits such as \reff{martlim} are well-known, see for example
  Section III.7 in \cite{AN72}.  However,  as the immigration mean if
  infinite since $g'_I(1-)=+\infty $, we
  deduce that  in  our   setting   $\rE[X_t]=\infty$. We    have    not
  found results such as \reff{martlim} 
  in    the 
  literature. 
\end{rem}

\begin{rem}\label{RemLinnik} 
 According to \reff{Zlaplace01} and \reff{Zlaplace02}, one can check that
\begin{equation}
   \label{eq:LaplaceW}
\rE[\expp{-\lambda W}]=\left( 1 +\lambda^{b-1}\right)^{-\frac{b}{b-1}}
=\rE[\expp{-\lambda^{b-1} G}],
\end{equation}
where $G$ has the  $\Gamma(\frac{b}{b-1},1)$ distribution. For $b=2$,
one get that $W$ is $\Gamma(1, 2)$. For $b\in (1, 2)$, according to
Proposition 1.5  in \cite{BBS08}, using notations  from Propositions 4.2
and 4.3 in \cite{J10}, we get  that $W$ is distributed as $\chi_{b-1,b}$
and  thus  has  a generalized  positive Linnik  distribution  with
parameter $(b-1, b)$ see the first paragraph of Section 2.3 in
\cite{J10} and the  references therein. 
Remark 2.2 and (2.25)  in \cite{J10} give that $W$  has intensity $f_{b-1, b}$ on
$(0, +\infty )$, where for $\ra\in (0, 1)$, $\rbb>0$ and $z>0$:
\begin{equation}
   \label{eq:def-fbb}
f_{\ra, \rbb}(z)=\frac{1}{\pi}\int_0^{\infty} \frac{\expp{-zy}\sin (\pi
  \rbb F_{\ra
  }(y))}{[y^{2\ra}+2y^{\ra}\cos(\ra\pi)+1]^{\frac{\rbb}{2\ra}}}\, dy,
\end{equation}
with 
\[
F_{\ra}(y)=1-\frac{1}{\pi\ra}\cot^{-1}\left(\cot(\pi \ra)+\frac{y^{\ra}}{\sin
    (\pi \ra)}\right). 
\]
We also give another representation of the density of $W$ using the fact
that in our case $\rbb=\ra+1=b$. Indeed,
according to (4.7), Proposition 4.3 (iii) in \cite{J10} we have that:
\begin{equation}\label{eq:james}
f_{\ra, \ra+1}(z)= -z  f'_{\ra, 1}(z)
\end{equation}
using the representation of $f_{\ra, 1}$ from Proposition 2.8 
and (2.22):
\[
f_{\ra, 1}(z)=\int_0^\infty \frac{\expp{z/y} }{y} \Delta_{\ra, 1} (y)\, dy
\quad\text{with}\quad 
\Delta_{\ra, 1} (y) = 
\frac{1}{\pi} \frac{\sin(\pi ( 1-
  F_{\ra}(y)))}{[y^{2\ra}+2y^{\ra}\cos(\ra\pi)+1]^{\frac{1}{2\ra}}} \cdot
\]
\end{rem}

\begin{rem}
Let $\chi$ be the  GW process introduced in Remark
\ref{rem:size-biased}. 
Recall from \cite{AN72} Formula (4) p.~108 that
$\E[\chi_t]=\expp{\frac{t}{b-1}}$. Let $W'=\lim_{t\to+\infty}
\expp{-\frac{t}{b-1}}\chi_t$.  By \cite{J10} Proposition 4.1 and
Proposition 4.3 (iii), the distribution of $W'$ has density
$-f'_{b-1,1}$.  
Then, Equation \eqref{eq:james} readily implies that the distribution of
$W$ is the size-biased distribution of $W'$ i.e., for every bounded
continuous function $\varphi$, we have: 
\begin{equation}\label{eq:size-biased-W}
\rE[\varphi(W)]=\rE[W'\varphi(W')].
\end{equation}

Another way of getting this identity is to use the relationship between
the processes $X$ and $\chi$. For every bounded continuous function
$\varphi$, we have: 
\[
\rE\left[\varphi(\expp{-\frac{t}{b-1}}X_t)\right]=\rE\left[\frac{\chi_t}{\rE[\chi_t]}\varphi(\expp{-\frac{t}{b-1}}(\chi_t-1))\right].
\]
Moreover, the expression of $f_{b-1,1}$ implies that the variable $W'$ admits every moment of order $\theta<b$. Then the martingale $(\chi_t/\rE[\chi_t],t\ge 0)$ is uniformly integrable (see \cite{AR04} for this result for a dicrete time GW process) and taking the limit in the previous equation gives \eqref{eq:size-biased-W}.
\end{rem}

\subsection{Proof of Theorem \ref{thm:gwi}}
\label{sec:pgwi}

We first  prove two  intermediate lemmas,  the first  one on  the Markov
property  for the  ancestral process  $M^0$ and  the second  one on  the
distribution of the  jumping times of $\tilde M$.   Recall the notations
of   Subsection   \ref{sec:notations_death}   for  the   jumping   times
$(\tau_n, n\geq  0)$ and  the jumping  sizes $(\xi_n,  n\geq 0)$  of the
ancestral process $M^0$, see \reff{tau0}, \reff{taun} and \reff{jumpsize}.

\begin{lem}\label{lem:dist_jumps} We set $(\mathcal{G}_n, \, n\ge 0)$ the filtration generated by the process $((M_{-\tau_{n}}^0,\tau_{n}),n\ge 0)$.
Under $\bP$, for every $n>0$, conditionally given
$\mathcal{G}_{n-1}$, the random variables $\xi_n$ and
$\tau_n$ are independent. Moreover, the conditional distribution of the
random variable $\xi_n$ given $\mathcal{G}_{n-1}$  has 
generating function  $g_{[M_{-\tau_{n-1}}^0]}$, with $g_{[\cdot]}$
defined  in \eqref{eq:gf_jumps2}.
\end{lem}

\begin{proof}
According to Remark 5.6 in \cite{CD12}, we have for $t>0$:
\begin{equation}\label{eq:xi_0}
\bE [r^{\xi_0}\bigm|\tau_0=t]=1-(1-r)^{b-1}=g_I(r).
\end{equation}
Thus $\xi_0$ and $\tau_0$ are independent.
Then by the branching property, it suffices to study the case $n=1$. Let us first compute the conditional distribution of $\tau_1$.

Recall that $\tau_1=\tau_1^I\vee \tau_1^B$. By standard properties of Poisson point measures, we have
\begin{align}
\bP (\tau_1^I<u|\tau_0=t,M_{-\tau_0}=n) & =\exp\left\{-\int_u^tds\, \int_{(0,+\infty)}r\pi(dr)\P_r (H(\ct)>s)\right\}\nonumber\\
& =\exp\left\{-\int_u^tds\int_{(0,+\infty)}r\pi(dr)\left(1-\expp{-r\N[H(\ct)>s]}\right)\right\}\nonumber\\
& =\exp\left\{-\int_u^tds\, \phi(c(s))\right\}\nonumber\\
& =\exp\left\{-b\int_u^t\frac{\alpha\, ds}{\expp{(b-1)\alpha s}-1}\right\}\nonumber\\
& =\left(\frac{\expp{\alpha  t}c(t)}{\expp{\alpha u}c(u)}\right)^b.
\label{eq:dist_tau_1_I}
\end{align}
Moreover, by Theorem 2.7.1 of \cite{DLG02}, we have, using  $\tilde
\psi(\lambda)=\psi(\lambda)/\lambda$: 
\begin{equation}\label{eq:dist_tau_1_B}
\bP (\tau_1^B<u|\tau_0=t,M_{-\tau_0}=n)=\left(\frac{\tilde\psi(c(t))}{\tilde\psi(c(u))}\right)^n.
\end{equation}
Recall that $-c'(u)=\psi(c(u))=\alpha c(u)+\gamma c(u)^b$. We deduce the conditional distribution of $\tau_1$:
\begin{multline*}
\bP (\tau_1\in du|\tau_0=t,M_{-\tau_0}=n)\\
\begin{aligned}
&= \bP ( \tau_1^B\in du, \, \tau^I_1<u |  \tau_0=t,
M_{-\tau_0}=n) +  \bP ( \tau_1^I\in du, \, \tau^B_1<u |  \tau_0=t,
M_{-\tau_0}=n)\\
   & = \bP
     (\tau_1^B\in du|\tau_0=t,M_{-\tau_0}=n)
\bP (\tau_1^I<u|\tau_0=t,M_{-\tau_0}=n)\\ 
&\hspace{4.5cm} +\bar\P (\tau_1^I\in du|\tau_0=t,M_{-\tau_0}=n)
\bP (\tau_1^B<u|\tau_0=t,M_{-\tau_0}=n)\\
& = \left(\frac{\expp{\alpha  t}c(t)}{\expp{\alpha u}c(u)}\right)^b 
\left(\frac{\alpha+\gamma c(t)^{b-1}}{\alpha+\gamma
    c(u)^{b-1}}\right)^n \left[
\,\frac{n\gamma(b-1)(-c'(u))c(u)^{b-2}}{\alpha+\gamma
  c(u)^{b-1}}
+ b\left(-\alpha+\frac{(-c'(u))}{c(u)}\right)
\right] du\\
& =\left(\frac{\expp{\alpha  t}c(t)}{\expp{\alpha
      u}c(u)}\right)^b\left(\frac{\alpha+\gamma
    c(t)^{b-1}}{\alpha+\gamma
    c(u)^{b-1}}\right)^n
\left[n\gamma(b-1) c(u)^{b-1}+b\gamma c(u)^{b-1}\right]du\\ 
& =\left(\frac{\expp{\alpha  t}c(t)}{\expp{\alpha u}c(u)}\right)^b
\left(\frac{\alpha+\gamma c(t)^{b-1}}{\alpha+\gamma
  c(u)^{b-1}}\right)^n\gamma(nb+b-n)c(u)^{b-1}du .
\end{aligned}
\end{multline*}
We deduce that:
\begin{align*}
\bP ( \tau_1^B\in du, \, \tau^I_1<u | \tau_1\in du, \tau_0=t,
M_{-\tau_0}=n)
&= \frac{n(b-1)}{nb+b-n},\\
 \bP ( \tau_1^I\in du, \, \tau^B_1<u | \tau_1\in du, \tau_0=t,
M_{-\tau_0}=n)
&= \frac{b}{nb+b-n}\cdot   
\end{align*}
Using formulas \eqref{eq:xi_1^B}, \eqref{eq:xi_1^I}, 
\reff{eq:g_stable}, the expression of $\phi$,  and the definition \reff{eq:gf_jumps2} of $g_{[n]}$, we get that 
\begin{align*}
\bE [r^{\xi_1}|\tau_0=t,\xi_0=n,\tau_1=u] 
& =
g_t(t-u,r) \frac{n(b-1)}{nb+b-n}
+\left(1-\frac{\phi((1-r)c(u))}{\phi(c(u))}\right)\frac{b}{nb+b-n}
\\
& =g_B(r)\frac{n(b-1)}{nb+b-n}+g_I(r)\frac{b}{nb+b-n}
\\
&= g_{[n]} (r). 
\end{align*}
Since the  latter expression  does not  depend on  $u$, this  proves the
conditional  independence between  $\xi_1$  and  $\tau_1$. Moreover,  we
indeed   recover  the   expression  of   \eqref{eq:gf_jumps1}  for   the
conditional generating function of $\xi_1$.
\end{proof}

\begin{rem}
\label{rem:indep-js}
Lemma \ref{lem:dist_jumps} implies in particular the independence between the first jumping time $\tau_0$ and the states of $M^0$, i.e. the sequence $(M^0_{-\tau_n},n\ge 0)$.
\end{rem}

\medskip  We  denote,  for  every  $n\ge 0$,  the  scaled  jumping  time
$\tilde\tau_n=R(\tau_n)$     and     the    corresponding   time intervals
$\Delta_n=\tilde\tau_n-\tilde\tau_{n-1}$     with     the     convention
$\tilde\tau_{-1}=0$. The  following lemma gives the  distribution of the
time intervals given the states of the process $\tilde M$.

\begin{lem}\label{lem:dist_Delta}
Conditionally given $(\tilde M_{\tilde \tau_n},n\ge 0)$, the random
variables $(\Delta_n,n\ge 0)$ are independant with for all $u\geq 0$ and $n\geq 0$:
\begin{equation}
   \label{eq:lawD}
\bP \bigl(\Delta_n>u\bigm|(\tilde M_{\tilde \tau_k},k\ge 0)\bigr)=\exp\left(-\left(\tilde M_{\tilde\tau_n}+\frac{b}{b-1}\right)u\right).
\end{equation}
\end{lem}

\begin{proof}
Let us first compute the distribution of $\tilde \tau_0=\Delta_0$.
For every $u\ge 0$, we have:
\[
\bar\P \left(\tilde\tau_0>u\bigm|(\tilde M_{\tilde\tau_k},k\ge 0)\right)
 =\bar\P \left(R(\tau_0)>u\bigm| (M_{-\tau_k}^0,k\ge 0)\right)
 =\bar\P \bigl(\tau_0<R^{-1}(u)\bigr),
\]
using the independance between $\tau_0$ and the states of $M^0$, see
Remark \ref{rem:indep-js}, and that
$R$ is non-increasing. 
\medskip

By the branching property, for every $r>0$, conditionally on 
$Z_{-r}$, the random variable $M_{-r}^0$ is distributed under $\bar\P $
according to a Poisson distribution with parameter
$c(r)Z_{-r }$. We get:
\[
\bP \bigl(\tau_0<r)=\bP(M_{-r}^0=0)=\bE\left[\expp{-c(r)Z_{-r}}\right].
\]
Note  that $Z$ is a CBI, so that \reff{Zlaplace01} holds (with $Y$
distributed as $Z$). Thus, using  \reff{Zlaplace02} as well as
\reff{eq:CR-1}, we deduce that: 
\[
\bP \bigl(\tilde \tau_0>u)  =\bu\left(c\left(R^{-1}(u)\right)\right)\\
=\expp{-\frac{bu}{b-1}},
\]
which is the looked after expression since $\tilde M_0=0$.

\medskip
Let us now compute the distribution of $\Delta_1=\tilde \tau_1 - \tilde \tau_0$.
First, using that $\tau_n=\tau_n^I\vee \tau_n^B$ (see \reff{eq:t=min})
and Equations \eqref{eq:dist_tau_1_I} and \eqref{eq:dist_tau_1_B}, we have
\begin{align*}
\bar\P \left(\tilde\tau_1>u\bigm| \tilde \tau_0=t,\tilde M_{\tilde\tau_0}=k\right) & =
\bar\P \left(\tau_1<R^{-1}(u)\bigm| \tau_0=R^{-1}(t),\ M_{-\tau_0}^0=k\right)\\
& = \bar\P \left(\tau_1^I<R^{-1}(u)\bigm| \tau_0=R^{-1}(t),\
  M_{-\tau_0}^0=k\right)\\ 
& \hspace{2cm} \times \bP \left(\tau_1^B<R^{-1}(u)\bigm| \tau_0=R^{-1}(t),\
  M_{-\tau_0}^0=k\right) 
\\
& =\frac{\expp{\alpha bR^{-1}(t)}c\bigl(R^{-1}(t)\bigr)^b}{\expp{\alpha bR^{-1}(u)}c\bigl(R^{-1}(u)\bigr)^b}\frac{\tilde\psi\Bigl(c\bigl(R^{-1}(t)\bigr)\Bigr)^k}{\tilde\psi\Bigl(c\bigl(R^{-1}(u)\bigr)\Bigr)^k}\cdot
\end{align*}
Using the expressions of $R$ and \reff{eq:CR-1}, we have
\[
\tilde\psi\Bigl(c\bigl(R^{-1}(t)\bigr)\Bigr)  =\alpha\expp{ t}
\quad\text{and}\quad 
\expp{\alpha R^{-1}(t)}  =\left(\frac{\expp{t}}{\expp{t}-1}\right)^{\frac{1}{b-1}}.
\]
This and \reff{eq:CR-1} again give:
\[
\bP\left( \Delta_1> u \bigm| \tilde \tau_0=t,\tilde M_{\tilde\tau_0}=k\right)
=\bP \left(\tilde\tau_1>u+t\bigm| \tilde \tau_0=t,\tilde M_{\tilde\tau_0}=k\right)=\expp{-\left(k+\frac{b}{b-1}\right)u}.
\]

By an easy induction, Lemma \ref{lem:dist_jumps} implies that, conditionally given $\mathcal{G}_0$, the random variable $\tau_1$ is independent of the states $(\tilde M_{\tilde \tau_k},k\ge 1)$. Therefore, we get
\[
\bP\left( \Delta_1> u \bigm| \tilde \tau_0=t,(\tilde M_{\tilde\tau_n},n\ge 0)\right)=\expp{-\left(\tilde M_{\tilde\tau_0}+\frac{b}{b-1}\right)u}.
\]

The proof then follows by induction and by the Markov property.
\end{proof}

\begin{proof}[Proof of Theorem \ref{thm:gwi}]
Lemmas \ref{lem:dist_jumps} and \ref{lem:dist_Delta}  imply the Markov
property for the process $\tilde M$, Lemma \ref{lem:dist_Delta} gives
the transition rates and Lemma \ref{lem:dist_jumps} gives the
distribution of the jumps. This and 
\reff{rateX}, \reff{eq:gf_jumps1} and \reff{eq:gf_jumps2} 
give the result. 
\end{proof}

\subsection{Distribution of the sizes of the families of the current population}

Recall   the   forest   $\bar\cf=(h_i,\cf_i)_{i\in  I}$   from   Section
\ref{sec:random-forest}   and   the   process   $Z$   from   Proposition
\ref{prop:Z=cbi}. Let us denote by
\[
I_0=\{i\in I,\ h_i<0\mbox{ and }\ell^{-h_i}(\cf_i)\ne 0\}
\]
the immigrants that have descendants at time 0. We order the set $I_0$
by the date of arrival of the immigrant: $I_0=\{i_k, k\ge 0\}$ 
with $-\tau_0=h_{i_0}<h_{i_1}<h_{i_2}<\cdots<0$.
For every $k\ge 0$, we set $\zeta_k$ the size of the population at time
0 generated by the $k$-th immigrant, that is:
\[
 \zeta_k=\langle
\ell^{-k_{i_k}}(\cf_{i_k}),1\rangle.
\] 
Notice that $\sum_{k=0}^{+\infty}\zeta_k=Z_0$.
\medskip

Let $\{\sigma_t: t\geq0\}$ be a $(b-1)$-stable subordinator:
$\rE[\expp{-x \sigma_t}]=\expp{-tx^{b-1}}$. Recall $\kappa$ defined in
\reff{Zlaplace02}. 

\begin{prop}\label{PRMage}
Consider the sub-critical stable branching mechanism
  with immigration \eqref{eq:def-psi-b-s}. 
The random point measure $\sum_{k\in \N}\delta_{\kappa\zeta_k}(dx)$ is 
 a Poisson point  measure on $[0, \infty)$ with intensity $
 g( x)\, dx$ where for $x>0$:
\beqlb\label{Fung}
g(x)=\frac{b}{x}\rE\left[\expp{-
    \left({x}/{\sigma_1}\right)^{b-1}}\right].
\eeqlb
We also have that for all $\lambda\geq 0$:
\beqlb
\label{gx}
\int_0^{\infty}(1-\expp{-\lambda x})\,g(x)\, dx=\frac{b}{b-1}\, \log
\left(1+\lambda^{b-1}\right)
= -\log \left( \bu (\kappa \lambda )\right). 
\eeqlb
\end{prop}

\begin{proof}
  Recall    the    GWI process   $X$   from    Theorem    \ref{thm:gwi}.    Let
  $\{0=T_0<  T_1<T_2<\cdots\}$ be  the  immigration times  of $X$  which
  forms a  Poisson process with  rate $b/(b-1)$. Recall $g_B$  and $g_I$
  defined    in    \reff{eq:g_stable}    and    \reff{eq:def-gi}.    Let
  $\{X^i,  i\geq0\}$  be  a  sequence  of  independent  continuous  time
  Galton-Watson  processes  with  branching   rate  $1$  such  that  the
  offspring  law  has generating function  $g_B$  and  the law of
  $X^i_0$  has generating function 
  $g_I$. Then for $t\geq0$, we have:
\[
X_t=\sum_{T_i\leq t}X^i_{t-T_i}\quad\text{and}\quad
 W=\sum_{i\geq0}\expp{-T_i/{(b-1)}}W_i,
\]
where
$W_i\overset{a.s.}{=}\lim_{t\rightarrow\infty}\expp{-\frac{t}{b-1}}X^i_t$
so that $\{W_i: i\geq0\}$ are independent random variables with the same
distribution. 
By Theorem 3 on Page 116 of \cite{AN72}, we get that  for $x\geq0$,
$\rE[\expp{-x W_i}]=g_I(\varphi(x)) $, where $\varphi$ is a one-to-one
map form $[0, \infty )$ to $(0, 1]$ such that for $x\in (0, 1]$:
\[
\varphi^{-1}(x)=
(1-x)\exp\left\{\int_1^{x}\left(\frac{g_B'(1)-1}
    {g_B(r)-r}+\frac{1}{1-r}\right)dr\right\}.   
\]
We get that for $x\in (0, 1]$:
\[
\varphi^{-1}(x)=
(1-x)\exp\left\{\int_0^{1-x}\frac{u^{b-2}}{1-u^{b-1}}du\right\}
=\left(\frac{(1-x)^{b-1}}{1-(1-x)^{b-1}}\right)^{1/(b-1)}.
\]
This gives that for $x\geq 0$:
\[
\varphi(x)=1-\left(\frac{x^{b-1}}{1+x^{b-1}}\right)^{1/(b-1)}.
\]
We then deduce that:
\begin{equation}
   \label{eq:LaplaceWi}
\rE[\expp{-x W_i}]=g_I(\varphi(x))=1-(1-\varphi(x))^{b-1}=\frac{1}{1+x^{b-1}}\cdot
\end{equation}
This, Theorem \ref{thm:gwi} and Corollary \ref{cor:martlim}  imply that:
\begin{equation}
  \label{eqnlaw}
\left( \kappa \zeta_i ,
  i\in \N\right)\overset{d}{=}\left(\expp{-\frac{T_i}{(b-1)}}W_i,
  i\in \N\right). 
\end{equation}

Let $\{(\sigma_s^i, s\geq0), i\in \N\}$ be a sequence of independent
$(b-1)$-stable subordinators and $\{E^i,  i\in \N\}$ be a sequence of
independent exponentially distributed random variables with parameter
1. Then it is easy to see from \reff{eq:LaplaceWi} that 
\[
(W_i, i\in \N)\overset{d}{=} (\sigma_{E^i}^i, i\in \N)\overset{d}{=}
\left((E^i)^{\frac{1}{b-1}}\sigma_{1}^i, i\in \N\right),
\]
where the last equality follows  from scale invariant property of stable
subordinator.  Thus  we have: 
\begin{equation}
   \label{eqnlaw01}
\left(  \kappa \zeta_i,
  i\in \N\right)\overset{d}{=}\left(\expp{-\frac{T_i}{(b-1)}}W_i        ,
  i\in \N\right)\overset{d}{=}
\left(\expp{-\frac{T_i}{(b-1)}}(E^i)^{\frac{1}{b-1}}\sigma_{1}^i       ,
  i\in \N\right).  
\end{equation}
   On   the    other   hand,    notice   that
$\sum_{i}\delta_{T_i}(dt)\delta_{E^i}(dx)$ is  a Poisson  random measure
on $[0,\infty)^2$ with  intensity $\frac{b}{b-1}dt\, \expp{-x}dx$.  Thus
$\sum_{i}\delta_{\{\expp{-T_i}E^i \}}(ds)$  is a Poisson  random measure
on $[0,\infty)$ with intensity $\frac{b}{b-1}s^{-1}\expp{-s}ds$. Indeed,
for any bounded  positive measurable function $f$ on  $[0, \infty)$, one
has
\begin{align*}
\rE\left[\expp{-\sum_{i} f(\expp{-T_iE^i})}\right]
&=\exp\left\{-\int_{0}^{\infty}\int_{0}^{\infty}(1-\expp{-f(\expp{-t}x)})
  \frac{b}{b-1}dt \, \expp{-x}dx\right\}\\
&= \exp\left\{-\int_{0}^{\infty}\int_{0}^{\infty}(1-\expp{-f(s)})
  \frac{b}{b-1}dt \expp{t} \expp{-s\expp{t}}ds\right\}\\ 
&= \exp\left\{-\int_{0}^{\infty}(1-\expp{-f(s)})
  \frac{b}{b-1}s^{-1}\expp{-s}ds\right\}, 
\end{align*}
where the last equality follows from $ \int_{0}^{\infty}
\expp{t-s\expp{t}}dt=\int_{1}^{\infty}  \expp{-st}dt=s^{-1}\expp{-s}$. 
 Hence, we deduce that
\[
F(ds\,dx):=\sum_{i}\delta_{\{ \expp{-T_i}E^i \}}(ds)\delta_{\sigma_1^i}(dx)
\]
 is a Poisson point  measure on $[0,\infty)^2$ with intensity
 $\frac{b}{b-1} s^{-1} \expp{-s} ds \, \rP(\sigma_1\in dx)$. Define
$$
G(ds)=\sum_i
\delta_{\{\expp{-\frac{T_i}{b-1}}(E^i)^{\frac{1}{b-1}}\sigma_{1}^i
  \}}(ds). 
$$
We shall prove that $G$ is a Poisson point measure on $[0,\infty)$ with
intensity $g(s)ds$. We only need to identify  the intensity measure. For
any positive measurable function $f$ on $[0,\infty)$, we have:
\begin{align*}
 \log \rE\left[\exp\left\{-\int_0^{\infty} f(s)G(ds)\right\}\right]
&=
\log \rE\left[\exp\left\{-\int_{[0,\infty)^2} f( s^{1/(b-1)}x) \,
  F(dsdx)\right\}\right]\\ 
&= -\int_{[0,\infty)^2}\left(1-\expp{-f( s^{1/(b-1)}x)}\right)
  \frac{b}{b-1}\expp{-s} \frac{ds}{s}\, \rP(\sigma_1\in dx)\\
&=-\int_{[0,\infty)^2}\left(1-\expp{-f(t)}\right)
  \frac{b}{t}\expp{-(t/x)^{b-1}}dt\, \rP(\sigma_1\in dx)\\
&=-\int_{[0,\infty)}(1-\expp{-f(t)})g(t)\,dt.
\end{align*}
Then the desired result follows. We now prove the last part of the
proposition. We have:
\begin{align*}
\int_{[0,\infty)}(1-\expp{-\lambda t})\,g(t)\,dt
&= \frac{b}{b-1}\int_0^{\infty}s^{-1}\expp{-s}ds \int_0^{\infty}\rP(\sigma_1\in
  dx)(1-\expp{-\lambda s^{1/(b-1)}x})\\
&=\frac{b}{b-1}\int_0^{\infty} ds \, s^{-1}\expp{-s}\left(1-\expp{-s
  \lambda^{b-1}}\right)\\
&= \frac{b}{b-1}\log \left(1+\lambda^{b-1}\right).
\end{align*}
\end{proof}

\begin{rem}\label{RemSub}
From the proof of Proposition \ref{PRMage}, we have
\[
\left(\sum_{i=0}^{k}\expp{-\frac{T_i}{b-1}}W_i, k\in \N\right)\overset{d}{=}
\left(\sum_{i=0}^k
\expp{-\frac{T_i}{b-1}}(E^i)^{\frac{1}{b-1}}\sigma_{1}^i,
k\in \N\right)\overset{d}{=} \left(\sigma_{S_k}, k\in \N\right),  
\]
where $S_k=\sum_{i=0}^k \expp{-T_i}E^i$. 
Notice that  $(S_k, k\geq 0)$ is
independent of $(\sigma_s, s\geq0)$. 
Since $\sum_{i}\delta_{\{ \expp{-T_i}E^i \}}(ds)$
 is a Poisson point  measure on $[0,\infty)$ with intensity
 $\frac{b}{b-1}\,  s^{-1} \expp{-s} \, ds$, we get that 
 $
\{\expp{-T_i}E^i,  i\in \N \}$ are the jump sizes of  a Gamma subordinator
$\left(\Gamma_t, \, t\in [0,\frac{b}{b-1}]\right)$ with L\'evy
measure ${s}^{-1}\expp{-s}\, ds$. And we recover that $S_\infty $ is
distributed as $\Gamma_{\frac{b}{b-1}}$ and is thus
$\Gamma\left(\frac{b}{b-1}, 1\right)$-distributed (see also Remark
\ref{RemLinnik}). 
 Therefore, we get that
 $\{\kappa \zeta_i, i\in \N\}$ are the jump sizes  of $\{\sigma_{\Gamma_t},
 0\leq t\leq \frac{b}{b-1}\}$. This induces a Poisson-Kingman partition;
 see \cite{P03}. 

\end{rem}

The distribution of $(\zeta_k, k\geq 0)$ is related to the
Poisson-Dirichlet distribution in the quadratic case. Recall that $Z_0=
\sum_{k\in \N} \zeta_k$.  
\begin{cor}
\label{propPD} 
Consider the sub-critical quadratic  branching mechanism   with immigration 
\eqref{eq:def-psi-b} with $b=2$. 
  Let $(\zeta_{(k)}, k\in \N)$  be  the
decreasing  order statistics of $(\zeta_k, k\in \N)$. Then, the
random sequence $\left( \zeta_{(k)}/Z_0 , k\in \N\right)
$ has a Poisson-Dirichlet distribution with parameter $2$.
\end{cor}
\begin{proof}
When $b=2$, we have   $\sigma(t)=t$. Then $\{\kappa \zeta_k, k\in \N\}$
are jump sizes  of $\{{\Gamma_t}: 0\leq t\leq 2\}$. The result follows
from Proposition 5 in \cite{PY97}, see also \cite{K75}. 
\end{proof}

\begin{rem}
\label{rem:rep-quad}
  Assume  that $\psi(\lambda)=\alpha  \lambda+\gamma\lambda^2$. According
  to (\ref{eqnlaw})  above and Theorem  2.21 in \cite{F10}, we  have the
  following   so-called  GEM   representation:
the sequence  $\left( \zeta_k/Z_0 , k\in \N\right)$ is distributed as 
\begin{equation}
   \label{eqnlawstick}
\left(U_0,    (1-U_0)U_1,   \cdots,
    (1-U_0)\cdots(1-U_{k-1})U_{k},\cdots,   \right), 
\end{equation}
  where
  $\{U_i,  i\geq0\}$  are independent random variable with the same  Beta-$(1,  2)$  distribution;  see
  \cite{F10} and  references therein.  Moreover,  Corollary \ref{propPD}
  above  and  Theorem  2.7  in  \cite{F10}  give  that  the  size-biased
  permutation                                                         of
  $  \left(  \zeta_{(k)}/Z_0, k\in \N\right)  $
  also   has  the   same   law   as  the   family   of  age-ordered   in
  (\ref{eqnlawstick}). 
\end{rem}

When $b\in (1,2)$, it does not seem possible to get a result similar to
Corollary \ref{propPD} or Remark \ref{rem:rep-quad}, see
Remark \ref{rem:notB} below. 
\medskip

We consider the size-biased sample $V$  of $\left(\zeta_k/Z_0,
  k\in \N\right)$ under $\bP$. Let $K$ be a $\N$-valued random variable
such that,  conditionally on $\left(\zeta_k/Z_0,
  k\in \N\right)$,  $K$ is equal to $k$ with probability
$\zeta_k/Z_0$. Then, $V$  is distributed as $\zeta_K/Z_0$ under
$\bP$:
\[
\rP(V\in dx)=\sum_{k\geq 0}  x\, 
\bP(\zeta_k/Z_0 \in dx).
\] 
We shall also consider  the size-biased sample $\zeta^*$  of
$\left(\zeta_k, 
  k\in \N\right)$, which is  distributed as $\zeta_K$. 

Recall that $f_{b-1, b}$ defined in \reff{eq:def-fbb} is the density of
$\kappa Z_0$. 
Then with Proposition \ref{PRMage} and Remark \ref{RemSub} in hand,
Theorem 2.1 of \cite{PPY92} implies that the distribution of $V$ and
$\kappa \zeta^*$ have densities  given by: 
\[
f_V(x) =  x \int_0^{\infty}{ t g( xt)}f_{b-1,
  b}((1-x)t)\,  dt\quad\text{for $x\in (0, 1)$},
\]
  and
\[
f_{\kappa \zeta^*} (x)= x {g( x)}  \int_x^{\infty}
f_{b-1,b}(t-x)\, \frac{dt}{t} \quad \text{for $x>0$}.
\]
   See also (25) and (19) in Section 3 of \cite{P03}. In the following
   proposition we characterize the law of $\zeta^*$ via its Laplace
   transform and compute the moments of $V$. Recall $\bu$ from
   \reff{Zlaplace02}. We set:
\[
G=-\frac{\bu'}{\bu}\cdot
\]

\begin{prop} \label{prop:moments}
Consider the sub-critical stable branching mechanism   with immigration
\eqref{eq:def-psi-b-s}. We 
have for $\lambda\geq 0$,
\begin{equation}
\label{YL}
\rE[\expp{-\lambda \zeta^*}]=\int_0^{\infty}G(\lambda+\mu) \bu (\mu)\, d\mu,
\end{equation}
and for $n\geq 1$:
\begin{equation}
\label{XN}
\rE[V^n]=\int_0^{\infty} v_n(t)\, dt
\quad\text{with}\quad v_n(t) =(-1)^n \frac{t^n}{n!}\,
{G^{(n)}(t)}\bu(t). 
\end{equation}
\end{prop}

\begin{proof} First,  by property of Poisson point  measure and
  \reff{gx}, we get:
\begin{align*}
   \bE\left[\sum_{i=0}^{+\infty}\zeta_i\expp{-\lambda \zeta_i-\mu Z_0}\right] 
&= \partial_\lambda  
\left(
\partial_\rho \, 
\bE\left[\expp{-\rho\sum_{i=0}^{+\infty}\expp{-\lambda \zeta_i}-\mu
    Z_0}\right]\right)_{\mid{\rho=0}} \\
&= - \exp\left\{-\int_0^{\infty}(1-\expp{-\mu x})\kappa g(\kappa x)\, dx\right\}
\partial_\lambda 
\int_0^{\infty}
\expp{-(\mu+\lambda) x}\kappa g(\kappa x)\,dx \\
&= \exp\left\{-\int_0^{\infty}(1-\expp{-\mu x})\kappa g(\kappa x)\, dx\right\}
 \int_0^{\infty}
x \expp{-(\mu+\lambda) x}\kappa g(\kappa x)\,dx \\
&= \bu(\mu) G(\lambda+\mu).
\end{align*}
Then \reff{YL} follows from
\[
\rE[\expp{-\lambda
  \zeta^*}]=\bE\left[\sum_{i=0}^{+\infty}\frac{\zeta_i}{Z_0}\expp{-\lambda 
    \zeta_i}\right]=\int_0^{\infty}d\mu
\, \bE\left[\sum_{i=0}^{+\infty}\zeta_i \expp{-\lambda \zeta_i-\mu
    Z_0}\right].  
\]
Next, observe that for $n\geq1$,
\[
    \bE\left[\sum_{i=0}^{+\infty}\zeta_i^ n \expp{-\mu
    Z_0}\right]
=  (-1)^{n-1}\left(
\partial^{n-1}_\lambda\,  
\bE\left[\sum_{i=0}^{+\infty}\zeta_i \expp{-\lambda \zeta_i-\mu
    Z_0}\right]\right)_{\mid{\lambda=0}}
= (-1)^{n-1}G^{(n-1)}(\mu) \bu(\mu). 
\]
We deduce that:
\begin{align*}
   \rE[V^n]
&=  \bE\left[\sum_{i=0}^{+\infty}
  \left(\frac{\zeta_i}{Z_0}\right)^{n+1}\right]   \\
&=  \int_{(0, +\infty )^n} dt_1 \ldots dt_n \, \ind_{\{0<t_1<t_2 \cdots
  <t_n\}} 
\int_{t_n}^{\infty}(-1)^{n}G^{(n)}(t) \bu (t)\, dt\\
&=  \int_0^{\infty}(-1)^n \frac{t^n}{n!} G^{(n)}(t)\, \bu (t)\, dt.
\end{align*}
This finishes the proof. 
\end{proof}

\begin{rem}
  The  moment  of $V$  in  \reff{XN}  can  be computed  explicitly.  Set
  $\eta=b-1$  and  $a=\gamma/\alpha$. Recall from 
\reff{Zlaplace02} that $
  \bu(t)=(1+ a t^\eta)^{-(\inv{\eta}+1)}$.  As $G^{(n-1)}(t)$  is  a  linear
  combination of  functions $ t^{k\eta  - n} (  1+ a t^\eta)^{-k}  $ for
  $k\in \{ 1, \ldots, n\}$, one gets that for $n\geq 1$:
\[
\lim_{t\rar0+} {t^n G^{(n-1)}(t)}\,\bu(t)=\lim_{t\rar+\infty}{t^n
  G^{(n-1)}(t)}\,  \bu(t)=0.
\]
Recall $v_n$ defined in \reff{XN}, so that  for $n\geq 0$ and $k\geq0$:
\[
\frac{v_n(t)}{(1+at^{\eta})^k}= 
(-1)^n  \frac{t^n}{n!} \frac{G^{(n)} (t)}{(1+ a t^\eta)^{k+1+1/\eta}}\cdot
\]
Then, for $n\geq1$ and $k\geq0$, by integration by parts, we have
\begin{multline}
\label{eq:rec-vn}
\int_0^{\infty}\frac{v_n(t)}{(1+at^{\eta})^k}\, dt\\
=
\left(1-\frac{\eta(k+1)+1}{n}\right)
\int_0^{\infty}\frac{v_{n-1}(t)}{(1+at^{\eta})^k}\, dt
+\frac{\eta(k+1)+1}{n} \int_0^{\infty}\frac{v_{n-1}(t)}{(1+at^{\eta})^{k+1}}dt, 
\end{multline}
and
\[
\int_0^{\infty}\frac{v_0(t)}{(1+at^{\eta})^k}\,dt
= \int_0^{\infty}\frac{(-\bu'(t))}{(1+at^{\eta})^k}\, dt
=\int_0^{\infty}\frac{a(1+\frac{1}{\eta})\eta t^{\eta-1}}{
  (1+at^{\eta})^{k+2+\frac{1}{\eta}}}dt
=\frac{\eta+1}{\eta(k+1)+1}\cdot
\]
The previous recursion formula gives the value of $\int_0^{\infty}
v_n(t)(1+at^{\eta})^{-k}\, dt$ for all $n\geq 0$ and $k\geq 0$. Using
\reff{eq:rec-vn} with $k=0$, we get: 
\begin{equation}
   \label{eq:XN1}
\rE[V^n]
=\left(1-\frac{(\eta+1)}{n}\right)\int_0^{\infty}v_{n-1}(t)\, dt
+ \frac{\eta+1}{n}\int_0^{\infty}\frac{v_{n-1}(t)}{1+at^{\eta}}\, dt.
\end{equation}
 In particular, one has:
\begin{align*}
\rE[V]
&=-\eta+ (\eta+1)\frac{\eta+1}{2\eta+1}=\frac{-{\eta}^2+{\eta}+1}{2{\eta}+1},\\
\rE[V^2]
&= (1-\frac{\eta+1}{2})E[X]-\frac{\eta(\eta+1)^2}{2\eta+1}
+\frac{(\eta+1)^2(2\eta+1)}{2(3\eta+1)}\\
&=  \frac{\eta^4-7\eta^3+\eta^2+7\eta+2}{2(2\eta+1)(3\eta+1)},\\
\rE[V^3]
&= \frac{23\eta^5-80\eta^4-30\eta^3+74\eta^2+43\eta+6}
{6(2\eta+1)(3\eta+1)(4\eta+1)}\cdot
\end{align*}
\end{rem}

\begin{rem}
\label{rem:notB}
  Based on  the moment  formulas above,  one can check  that $V$  is not
  Beta-distributed  if  $b<2$ (except maybe for one particuler value of $\eta$ ($\eta \simeq 0.428$) where the three first moments of $V$ coincide with those of a Beta distribution).  If   $b=2$,  then  according  to  Remark
  \ref{rem:rep-quad}, $V$ is Beta(1,2)-distributed.
\end{rem}

\section{Birth and death rates of the ancestral process}
\label{sec:bd-rate}

In this section, we assume that $\psi$ and $\phi$ are defined as
\reff{eq:def-psi} and \reff{eq:def-phi}, respectively. We assume that Conditions
 \reff{eq:grey-kappa} hold. Recall the function $\bu$ defined in
 \reff{Zlaplace01} and \reff{eq:def-bu0} which is the Laplace transform
 of $Z_0 $ under $\bP$. 
 Similar to the arguments on page 1330 in \cite{BD16}, see also
 Proposition 3.12 in \cite{CD12},  we have for $r\geq s>0$ and $x,y\in [0, 1]$:
\begin{equation}
   \label{eq:lap-M00}
\bE[x^{M^0_{-r}}y^{M^0_{-s}}]
=\bu(\lambda_0) \expp{\alpha (r-s)} \frac{\psi\Bigl(u\bigl( c(s)(1-y),
  r-s\bigr) \Bigr)}{ \psi\bigl(c(s)(1-y)\bigr)},
\end{equation}
with 
\begin{equation}
   \label{eq:lambda0}
\lambda_0= \lambda_0(x,y)= c(r)(1-x)+xu\bigl(c(s)(1-y), r-s\bigr).
\end{equation}

We first  summarize the results of the next two sections concerning the
quadratic case, see also \cite{BD16}. 
  \begin{rem}
\label{rem:death-rate-2}
In the quadratic case, $\psi(\lambda)=\alpha\lambda+\gamma\lambda^2$,
we have:
\[
c(t)=\frac{\alpha}{\gamma(\expp{\alpha t}-1)}
\quad\text{and}\quad 
\bu(\lambda)=\left(1+\frac{\gamma}{\alpha} \lambda\right)^{-2}.
\]
We get thanks to \reff{eq:lap-M00} and \reff{eq:lambda0} (taking $r=s=t$
and $x=y$) that for $t>0$ and $x\in [0, 1]$:
\[
\bE\left[x^{M^0_{-t}}\right]= \bu (c(t)(1-x))=\left(\frac{\expp{\alpha
    t} -1}{\expp{\alpha t} - x}\right)^2.
\]
We get  that for $n\geq 0$:
\[
\bP(M^0_{-t}=n)=(n+1) \expp{-\alpha t n} \left(1- \expp{-\alpha t}\right)^2.
\]
For the death rate, we deduce  from \reff{qmn} that for $n\geq 1$: $q^\rd_{n,m}(t)=0$ if $n-2\geq m\geq 0$ and
if $m=n-1$
\[
q^\rd_{n, n-1}(t)=n(\alpha+ \gamma c(t)). 
\]
For the birth rate, we deduce  from \reff{bmn} that for $n\geq 0$:
$q^\rbb_{n,m}(t)=0$ if 
$m\geq n+2 $ and
if $m=n+1$
\[
q^\rbb_{n, n+1}(-t)=(n+2)\gamma c(t). 
\]
\end{rem}

\subsection{Death process}
\label{sec:d-process}
Recall the ancestral process $M^0= ( M_{t}^0,  t< 0)$ defined in Section
\ref{sec:notations_death}. Notice that the branching property gives that the
ancestral process  is a Markov process. 
We first study the death rate of the time reversed ancestral process
$\hat M^{0}= ( M_{-t}^0,  t>0)$. 
Notice that $\hat M^{0}$ is a Markov process as the time reversal of a
Markov process. 

\medskip

 \begin{prop}
\label{prop:q-d}
Let   $\psi$   and   $\phi$   be  defined   by   \reff{eq:def-psi}   and
\reff{eq:def-phi} such  that conditions \reff{eq:grey-kappa}  hold.  The
process $\hat M^{0}$  is a c\`ad-l\`ag death process starting  at time 0
from $+\infty $ and with death rate given for $n>m\geq 0$ and $t>0$ by:
\begin{equation}
   \label{qmn} 
q^\rd_{n,m}(t)
=\lim_{\varepsilon\rightarrow0+}\inv{\varepsilon}
  \bP\left(M^0_{-(t+\varepsilon)}=m| M^0_{-t}=n\right) 
=\binom{n+1}{m}\, \frac{\left|\bar{u}^{(m)}\bigl(c(t)\bigr)\right|}
{ \left|\bar{u}^{(n)}\bigl(c(t)\bigr)\right|}\, 
\left| \psi^{(n-m+1)}\bigl(c(t)\bigr) \right|.
\end{equation} 
\end{prop}

There is no closed formula for the stable case unless it is
quadratic. However, the next lemma gives an explicit asymptotic for the
birth rate when the stable index $b$ goes down to 1. 

\begin{lem}\label{lem:ratecov}
 Consider     the    sub-critical     stable    branching     mechanism
  with immigration 
  \eqref{eq:def-psi-b-s}, that is  $\psi(\lambda)=\alpha\lambda+\gamma\lambda^b$
  with $\alpha>0$ and $b\in (1, 2]$. Then we have for $n>m\geq 0$ and $t>0$:
\[
\lim_{b\rightarrow1+}{q^\rd_{n, m}(t)}= \frac{n+1}{(n+1-m)(n-m)}
\, \frac{1}{\gamma  t}\cdot
\]
\end{lem}
\begin{proof} Recall that in the stable case (see \reff{Zlaplace02}): 
\[
\bar{u}(\lambda)=\left(1+\frac{\gamma}{\alpha}\lambda^{b-1}\right)^{-\frac{b}{b-1}},
\quad c(t)=\left(\frac{\alpha}{\gamma\left(\expp{(b-1)\alpha t}-1\right)}\right)^{\frac{1}{b-1}}.
\]
One can check  that $\bar{u}^{(n)}(\lambda)$, for $n\geq 1$, has the form:
\[
\bar{u}^{(n)}(\lambda)=\sum_{k=1}^{n}C_{b,k,n}\left(\frac{\gamma}{\alpha}\right)^k
\left(1+\frac{\gamma}{\alpha}\lambda^{b-1}\right)^{-\frac{b}{b-1}-k}
\, \lambda^{-(2-b)k - (n-k)},
\]
where $C_{b,k,n}$  are constants depending  only on  $b, k$ and  $n$ and
such  that $(-1)^n  C_{b,  k,  n}\geq 0$.  On  the  other hand,  writing
$ \lambda^{-b}$ as
$(\lambda^{b-1})^{-\frac{b}{b-1}}$, one  sees that, with the same
constants  $C_{b,k,n}$:
\[
\left(\frac{\alpha}{\gamma}\right)^{\frac{b}{b-1}} \, \left(\lambda^{-b}\right)^{(n)}
=\left(\left(\frac{\gamma}{\alpha} \lambda ^{b-1}\right)^{-\frac{b}{b-1}}\right)^{(n)}
=\sum_{k=1}^{n}C_{b,k,n}\left(\frac{\gamma}{\alpha}\right)^k\left(\frac{\gamma}{\alpha}
  \lambda^{b-1}\right)^{-\frac{b}{b-1}-k}\lambda^{-(2-b)k - (n-k)}.
\]
Since $\lim_{b\rightarrow 1+} c(t)=+\infty $, we get, as $b\rightarrow1+$, that:
\[
\left(1+\frac{\gamma}{\alpha}c(t)^{b-1}\right)^{-\frac{b}{b-1}}
\sim \left(\frac{\gamma}{\alpha}c(t)^{b-1}\right)^{-\frac{b}{b-1}} \, \expp{-\alpha b t}.
\]
Since the constants $C_{b,k,n}$ have all the same sign for given $n$, we    deduce that for $n
\geq1$,
as $b\rightarrow1+$:
\begin{align*}
\bar{u}^{(n)}(c(t))
&
  \sim\left(\frac{\alpha}{\gamma}\right)^{\frac{b}{b-1}}(\lambda^{-b})^{(n)}(c(t))
  \, \expp{-\alpha b t}\\ 
&=(-b)(-b-1)\cdots(-b-n+1)c(t)^{-b-n}\,
     \left(\frac{\alpha}{\gamma}\right)^{\frac{b}{b-1}} \expp{-\alpha b t}\\
&\sim (-1)^{n}n!\, c(t)^{-b-n} \,
     \left(\frac{\alpha}{\gamma}\right)^{\frac{b}{b-1}} \expp{-\alpha b t}.
\end{align*}
We deduce that:
\begin{align*} 
\lim_{b\to 1+}q^\rd_{n,m}(t)
&=
\lim_{b\to 1+}\binom{n+1}{m}\, \frac{\left|\bar{u}^{(m)}\bigl(c(t)\bigr)\right|}
{ \left|\bar{u}^{(n)}\bigl(c(t)\bigr)\right|}\, 
\left| \psi^{(n-m+1)}\bigl(c(t)\bigr) \right|
\cr
&=
\lim_{b\to 1+}\binom{n+1}{m}\, \frac{m!}
{ n!}c(t)^{n-m}
\left|  b(b-1)\cdots(b-n+m)c(t)^{b-n+m-1}\right|
\cr
&=\lim_{b\to 1+}\binom{n+1}{m}\, \frac{(n-1-m)!m!}
{ n!}(1-b)c(t)^{b-1}
\cr&=\frac{n+1}{(n+1-m)(n-m)}
\, \frac{1}{\gamma t}\cdot
\end{align*}

\end{proof}

\begin{proof}[Proof of Proposition \ref{prop:q-d}]
 The proof is divided in three steps. 

\medskip \noindent\textit{Step 1: Preliminary computations}. 

We set for $\lambda,\mu\in [0, 1]$ and $t,\varepsilon>0$:
\begin{align*}
   g_{t,\varepsilon}(\mu)
&= \expp{\alpha \varepsilon}  
\frac{\psi\Bigl(u\bigl( c(t)(1-\mu),
  \varepsilon\bigr) \Bigr)}{ \psi\bigl(c(t)(1-\mu)\bigr)},\\ 
\lambda^*_{t,\varepsilon}=\lambda^*_{t,\varepsilon}(\lambda, \mu)
&= c(t+\varepsilon)(1-\lambda)+\lambda u\bigl(c(t)(1-\mu),
  \varepsilon\bigr),\\
f^\rd_{t,\varepsilon}(\lambda, \mu)
&=\bu(\lambda^*_{t,\varepsilon}) g_{t,\varepsilon}(\mu),\\
f_0(\mu)
&=\bar{u}\bigl( c(t)(1-\mu)\bigr).
\end{align*}
Thanks to \reff{eq:lap-M00} and \reff{eq:lambda0}, we deduce that:
\begin{equation}
   \label{eq:def-f0}
f^\rd_{t,\varepsilon}(\lambda, \mu)=
\bE\left[\lambda^{M^0_{-(t+\varepsilon)}}\mu^{M^0_{-t}}\right]
\quad\text{and}\quad
f_0(\mu)
= f^\rd(1, \mu)=\bE\left[\mu^{M^0_{-t}}\right].
\end{equation}
We get for $n>m\geq 0$:
\begin{equation}
   \label{qnm1}
q^\rd_{n,m}(t)
=\lim_{\varepsilon\rightarrow0+}\frac{1}{\varepsilon}
\frac{\bP(M^0_{-(t+\varepsilon)}=m,   M^0_{-t}=n) }{\bP(M^0_{-t}=n)}
= \lim_{\varepsilon\rightarrow0+}\frac{1}{\varepsilon}
\frac{\partial^n_\mu \partial ^ m_\lambda \, f^\rd_{t,\varepsilon}(0,0)}{
  m! \, f_0^{(n)}(0)}\cdot
\end{equation}

First notice that for $n\geq 1$:
\begin{equation}
   \label{eq:deriv-f0}
f^{(n)}_0(0)=(-1)^n  \bu^{(n)} (c(t)) \, c(t)^n.
\end{equation}

We now study $\partial^n_\mu \partial ^ m_\lambda \,
f^\rd_{t,\varepsilon}(0,0)$ . We set:
\[
I_ {t,\varepsilon}(\mu)= \partial_\lambda
\, \lambda^*_{t,\varepsilon}(\lambda,\mu)
=u\bigl(c(t)(1-\mu),\varepsilon\bigr)-c(t+\varepsilon).
\]
Notice   that $g_{t,\varepsilon}(\mu)$  and $I_{t,\varepsilon}(\mu)$ are independent
of $\lambda$. 
We deduce that for $m\geq 0$:
\begin{equation}
   \label{qnm5}
\partial^m_\lambda \, f^\rd_{t,\varepsilon}(\lambda,
\mu)=\bu^{(m)}(\lambda^*_ {t,\varepsilon}) I_{t,\varepsilon}(\mu)^m
\,g_{t,\varepsilon}(\mu). 
\end{equation}
We also note that for $k\geq1$,
\begin{equation}
   \label{qnm6}
 I_{t,\varepsilon}^{(k)} (\mu)
=(-1)^kc(t)^k\, \partial_\lambda^k   u
  \bigl(c(t)(1-\mu), \varepsilon)\bigr)
\quad\text{and}\quad
\partial^k_\mu
\,  \lambda^*_{t,\varepsilon}(\lambda,\mu)
=\lambda\,  I_{t,\varepsilon}^{(k)} (\mu). 
\end{equation}
We deduce that for $k\geq 1$:
\begin{equation}
   \label{qnm3}
\partial^k_\mu\, \lambda^*_{t,\varepsilon}(0,0)=0
\quad\text{and}\quad
\partial^k_\mu \, \bu\bigl(\lambda^*_{t,\varepsilon}(\lambda, \mu)\bigr) _{\mid
   (\lambda,\mu)=(0,0) }=0. 
\end{equation}
We end this first step by a remark. We deduce from 
\begin{equation}
   \label{eq:dif-u-l}
\partial _\lambda u(\lambda,t)
=\frac{\psi(u(\lambda,t))}{ \psi(\lambda)}
\end{equation} 
(see \reff{eq:def-u}),  
Equations \reff{eq:dif-u-l} and 
\reff{eq:dif-u-t} and elementary computations that:
\beqlb\label{Ulim}
\partial^k_\lambda\,  u(\lambda, \varepsilon)=
\begin{cases}
1+o(1)& \text{if $k=1$}, \\
-\psi^{(k)}(\lambda)\, \varepsilon+o(\varepsilon)& \text{if $k\geq2$},
\end{cases}
\eeqlb
where   $\varepsilon$ goes down to 0, so that
$o(1)$ means a quantity which goes down to $0$ with
$\varepsilon$. 

\medskip \noindent \textit{Step 2: Study of $\partial^n _\mu \,(I_{t,\varepsilon}^m)(0)$.}

We now study  the value of $\partial^n  _\mu \,(I_{t,\varepsilon}^m)(0)$ for
$n\geq 0$  and $m\geq  0$ and  $(n,m)\neq (0, 0)$.  The case  $n>m=0$ is
trivial as  $\partial^n _\mu \,(I_{t,\varepsilon}^m)(0)=0$. We  have for all
$m>n\geq 0$:
\begin{equation}
   \label{eq:Ie=0}
I_{t,\varepsilon}(0)=0
\quad\text{and thus}\quad
\partial^n_\mu \,(I_{t,\varepsilon}^m)(0)=0. 
\end{equation}
For the case $m=1$, we deduce from  \reff{qnm6} and \reff{Ulim}  that
for $n\geq 1$: 
\begin{equation}
\label{Ilim}
{I^{(n)}_{t,\varepsilon}(0)}=\begin{cases}
x-c(t)+o(1)& \text{if $ n=1$}, \\
(-1)^{n+1}c(t)^n\, \psi^{(n)}\bigl(c(t)\bigr)\varepsilon+o(\varepsilon) &
\text{if  $n\geq2$}.
\end{cases}
\end{equation}
  For $n=m$,  Faa di  Bruno's formula,  $I_{t,\varepsilon}(0)=0$
(see \reff{eq:Ie=0}) and \reff{Ilim} give that:
\begin{equation}
\label{qnm8}
\partial^m_\mu\,( I_{t,\varepsilon}^m) (0)=m! (-1)^m
c(t)^{m} + o(1).
\end{equation}
We  shall prove    by  induction  over  $m\geq   1$ that for all $n>m\geq 1$:
\begin{equation}
   \label{eq:dernier}
\partial^n _\mu \, (I_{t,\varepsilon}^m)(0)
= \binom{n}{m-1} m! (-1)^{n+1} 
c(t)^n\, \psi^{(n-m+1)} 
\bigl(c(t)\bigr) \, \varepsilon + o(\varepsilon).
\end{equation}
 Thanks  to
\reff{Ilim},  we get  that  \reff{eq:dernier} holds  for  $m=1$ and  all
$n>m$.    Let us assume that
\reff{eq:dernier} holds for $m-1$ (and all  $n>m-1$), and let us prove
it holds for
$m$ (and all $n>m$).  We have for $n>m$:
\begin{align*}
   \partial^n _\mu\, (I_{t,\varepsilon}^m)(0)
&= \sum_{k=0}^n \binom{n}{k} I_{t,\varepsilon}^{(k)}(0)
  \,\, \partial_\mu^{n-k} (I^{m-1}_{t,\varepsilon})(0)\\
&= n I^{(1)}_{t,\varepsilon}(0) \, \partial^{n-1}_\mu
  (I^{m-1}_{t,\varepsilon})(0) + \binom{n}{m-1} 
  I_{t,\varepsilon}^{(n-m+1)}(0) \,  \partial_\mu^{m-1}\,
  (I^{m-1}_{t,\varepsilon})(0) + O(\varepsilon^2)\\
&= \left[n (m-1)! \binom{n-1}{m-2} + (m-1)! \binom{n}{m-1}\right]
(-1)^{n+1} c(t)^n\, \psi^{(n-m+1)} \bigl(c(t)\bigr) \varepsilon +o(\varepsilon) \\
&= \binom{n}{m-1} m! (-1)^{n+1} 
c(t)^n\, \psi^{(n-m+1)} 
\bigl(c(t)\bigr) \, \varepsilon + o(\varepsilon),
\end{align*}
where, for the second equality we used that $I_{t,\varepsilon}(0)=0$ (see
\reff{eq:Ie=0}) for the term $k=0$, then  $\partial^{n-k}_\mu\,
(I_{t,\varepsilon}^{m-1})(0)=0$ (see
\reff{eq:Ie=0}) for the terms $k>n-m+1$, and then $ I_{t,\varepsilon}^{(k)}(0)
  \,\, \partial_\mu^{n-k}\, (I^{m-1}_{t,\varepsilon})(0)=0(\varepsilon^2)$
  (see \reff{Ilim} and the induction hypothesis) for $n-m+1>k\geq 2$;
  and for the third equality
\reff{Ilim} (for $k=1$ and $k=n-m+1$), the induction hypothesis  and 
\reff{qnm8}. Thus \reff{eq:dernier} holds for all $n>m\geq 1$.

\medskip \noindent \textit{Step 3: Computation of $q^\rd_{n,m}(t)$.}

If we derive \reff{qnm5} $n$ times with respect to $\mu$ and
evaluate the derivative at $(0,0)$, we get  for $n>m\geq0$:
\begin{align}\label{qnm4}
\partial^n_\mu \partial^m_\lambda \, f^\rd_{t,\varepsilon} (0,0)
&
=\bu^{(m)}\bigl(c(t+\varepsilon)\bigr)\partial^n_\mu\,
  (I_{t,\varepsilon}^m \,g_{t,\varepsilon}) (0)\\
&
\nonumber
=\bu^{(m)}\bigl(c(t+\varepsilon)\bigr)\sum_{k=m}^n\binom{n}{k}\,
  g_{t,\varepsilon}^{(n-k)}(0)\, 
  \partial ^k_\mu \, (I_{t,\varepsilon}^m)(0),
\end{align}
where for the first equality we used that all the terms in Leibniz'
formula are 0 except one thanks to  \reff{qnm3}, and for the second  Leibniz'
formula  again with 
\reff{eq:Ie=0}. 
\medskip

Since 
$ g_{t,\varepsilon}(\mu)=\expp{\alpha\varepsilon} \partial_\lambda u\bigl(
c(t)(1-\mu), \varepsilon\bigr)$, see \reff{eq:dif-u-l}, we deduce from \reff{Ulim} that,  for
$k\geq 1$:
\begin{equation}
   \label{eq:lim-gke}
g_{t,\varepsilon}^{(k)}(0)=
(-1)^{k+1}c(t)^k\, \psi^{(k+1)}\bigl(c(t)\bigr)\,  \varepsilon+
o(\varepsilon). 
\end{equation}
This and \reff{Ilim}  imply that for $n-1>m\geq 0$:
\[
\sum_{k=m+1}^{n-1}\binom{n}{k}\, g_{t,\varepsilon}^{(n-k)}(0)\, \partial
  ^k_\mu\,( I_{t,\varepsilon}^m)(0)=O(\varepsilon^2)=o(\varepsilon).
\]

\medskip
Then, we deduce from \reff{qnm4} and \reff{eq:dernier}  that for $t>0$
and  $n>m\geq 0$ (with the convention that $\binom{n}{m-1}=0$ if $m=0$):
\begin{align}
\nonumber
\partial^n_\mu \partial^m_\lambda \, f^\rd_{t,\varepsilon} (0,0)
&=\bu^{(m)}\bigl(c(t+\varepsilon)\bigr)\left[\binom{n}{m}\partial ^m_\mu
\, (I_{t,\varepsilon}^m)(0) \, g_{t,\varepsilon}^{(n-m)}(0) +
  \partial ^n_\mu \,(I_{t,\varepsilon}^m)(0)\,
 g_{t,\varepsilon}(0)\right] + o(\varepsilon)\\
\nonumber
&=\bu^{(m)}\bigl(c(t)\bigr)\left[\binom{n}{m} + \binom{n}{m-1}
\right]
m! (-1)^{n+1} c(t)^n\,
\psi^{(n-m+1)} 
\bigl(c(t)\bigr) \,  \varepsilon + 
   o(\varepsilon)\\
\label{eq:deriv-fmn}
&=\binom{n+1}{m}\bu^{(m)}\bigl(c(t)\bigr)
m! (-1)^{n+1} 
c(t)^n\,\psi^{(n-m+1)} 
\bigl(c(t)\bigr) \,  \varepsilon + 
   o(\varepsilon).  
\end{align}
Notice  that  $o(\varepsilon)$  in  the  last  equality  is  uniform  on
$t\in [a,  b]$ far any given  $0<a<b<+\infty $. We then  deduce from the
latter equality, \reff{qnm1} and \reff{eq:deriv-f0} that for $n>m\geq0$:
\[
q^\rd_{n,m}(t)
=-\binom{n+1}{m}\frac{\bar{u}^{(m)}\bigl(c(t)\bigr)
\, \psi^{(n-m+1)}\bigl(c(t)\bigr)  }
{ \bar{u}^{(n)}\bigl(c(t)\bigr)}\cdot
\]
This finishes the proof. 
\end{proof}

\subsection{Birth process}
Recall that the ancestral process $M^0= ( M_{t}^0,  t< 0)$ defined in Section
\ref{sec:notations_death} is a Markov process thanks to the branching
property. 

\begin{prop}
\label{prop:q-b}

Let   $\psi$   and   $\phi$   be  defined   by   \reff{eq:def-psi}   and
\reff{eq:def-phi} such  that Conditions \reff{eq:grey-kappa}  hold. The
process $M^0$ is a c\`ad-l\`ag birth process starting  at time $-\infty $ from
$0$  and with birth rate given for $n>m\geq 0$ and $t>0$ by:
\begin{equation}
\label{bmn}
q^\rbb_{n,m}(-t)
=\lim_{\ez\rightarrow0+}\frac{1}{\ez}\bP
\left(M^0_{-(t-\ez)}=m|M^0_{-t}=n\right)
= \frac{(m+1)}{(m+1-n)!} \,
c(t)^{m-n}\, \left|\psi^{(m-n+1)}\bigl(c(t)\bigr) \right| .
\end{equation}
\end{prop}

Concerning the birth rate, it is possible to have an explicit formula in
the stable case. 
\begin{rem}
\label{rem:birth-rate-b}
  Consider     the    sub-critical     stable    branching     mechanism
  with immigration 
  \eqref{eq:def-psi-b-s}. Using \reff{Zlaplace02}, we deduce that for
  $m>n\geq 0$  and $t>0$:
\[
q^\rbb_{n,m}(-t)
=\frac{(m+1)} {(m-n+1)!}\, |
  b(b-1)\cdots(b-m+n)|\, \frac{\alpha}{\expp{(b-1)\alpha
    t}-1}\cdot
\]
\end{rem}

\begin{proof}[Proof of Proposition \ref{prop:q-b}]
We keep notations from the proof of
Proposition \ref{prop:q-d} for  $f^\rd_{t, \varepsilon}$ and $f_0$. 
We set for $\lambda,\mu\in [0, 1]$ and $t>\varepsilon>0$
$f^\rbb_{t,\varepsilon}(\lambda, \mu)
=f^\rd_{t-\varepsilon, \varepsilon} (\mu, \lambda)
$ for $\lambda,\mu\in [0, 1]$ and $t>\varepsilon>0$. 
Thanks to \reff{eq:def-f0}, we have  that:
\[
f^\rbb_{\varepsilon}(\lambda, \mu)=
f^\rd_{t-\varepsilon, \varepsilon} (\mu, \lambda)
=\bE\left[\lambda^{M^0_{-t+\varepsilon}}\mu^{M^0_{-t}}\right].
\]
Recall $f_0$ defined in \reff{eq:def-f0} and its derivative given by
\reff{eq:deriv-f0}.  
We get for $m>n\geq 0$:
\begin{equation}
   \label{bnm1}
q^\rbb_{n,m}(-t)
=\lim_{\varepsilon\rightarrow0+}\frac{1}{\varepsilon}
\frac{\bP(M^0_{-t+\varepsilon}=m,   M^0_{-t}=n) }{\bP(M^0_{-t}=n)}
= \lim_{\varepsilon\rightarrow0+}\frac{1}{\varepsilon}
\frac{\partial^n_\mu \partial ^ m_\lambda \, f^\rbb_{t,\varepsilon}(0,0)}{
  m! \,f^{(n)} _0(0)}\cdot
\end{equation}

\medskip
Since $
\partial^n_\mu \partial ^ m_\lambda \,
f^\rbb_{t,\varepsilon}(0,0)
 = 
\partial^m_\mu \partial ^ n_\lambda \,
f^\rd_{t-\varepsilon,\varepsilon}(0,0)$, using the continuity in
$\varepsilon$ of the function $c$, we deduce 
from \reff{eq:deriv-fmn}, and the fact that $o(\varepsilon)$ in
\reff{eq:deriv-fmn} is uniform in $t$ on any closed interval of $(0,
+\infty )$, 
 that for $m>n\geq 0$ and $t>\varepsilon>0$:
\begin{align*}
\partial^n_\mu \partial ^ m_\lambda \,
f^\rbb_{t,\varepsilon}(0,0)
&=\binom{m+1}{n}\bu^{(n)}\bigl(c(t-\varepsilon)\bigr)
n! (-1)^{m+1} 
c(t-\varepsilon)^m\,\psi^{(m-n+1)} 
\bigl(c(t-\varepsilon)\bigr) \,  \varepsilon + 
   o(\varepsilon)\\
&=\binom{m+1}{n}\bu^{(n)}\bigl(c(t)\bigr)
n! (-1)^{m+1} 
c(t)^m\,\psi^{(m-n+1)} 
\bigl(c(t)\bigr) \,  \varepsilon + 
   o(\varepsilon).
\end{align*}

We then deduce from the latter equality,  \reff{bnm1} and
\reff{eq:deriv-f0} that  for 
$n>m\geq0$:
\[
q^\rbb_{n,m} (-t)
=(-1)^{m-n+1} \frac{m+1}{(m+1-n)!} \,
c(t)^{m-n}\, \psi^{(m-n+1)}\bigl(c(t)\bigr)  .
\]
This finishes the proof. 
\end{proof}

\subsection{Bolthausen-Sznitman coalescent as limit of the ancestral process}
\label{sec:BS-coal}

The Bolthausen-Sznitman coalescent, $(\Pi(t), t\geq0)$, is a
continuous-time Markov chain taking values in the set of patitions of
$\N^*$. It can be easily defined by considering its restriction
$\Pi^{[n]}=(\Pi^{[n]}(t), t\geq0)$ to the set $[n]:=\{1,\,2,\,\cdots,
n\}$, for $n\geq1$. Denote by ${\mathcal P}_n$ be the set of partitions
of $[n]$. Then, the process  $\Pi^{[n]}$ is a continuous-time ${\mathcal P}_n$-valued Markov chain whose transition rates are as follows: if $\#\Pi^{[n]}(t)=k$, then any $m$ of the present blocks coalesce at rate
$$
\frac{(m-2)!(k-m)!}{(k-1)!},\quad 2\leq m\leq k\leq n,
$$
where $\#\Pi^{[n]}(t)$ denotes the number of blocks of $
\Pi^{[n]}(t)$. The Bolthausen-Sznitman
coalescent was first introduced in \cite{BS98}. It is also a member of the class of coalescents with
multiple collisions introduced in \cite{P99} and \cite{S99}. We refer to
the survey \cite{B09} for further results on coalescent processes.

Other   constructions  of   the  Bolthausen-Sznitman   appear  in   the
literature. See \cite{BLe00}  using the genealogy of  a continuous state
branching process (the corresponding branching mechanism corresponds in
some sense  to
the limit in \reff{eq:def-psi-b} as $b$ goes down to 1), \cite{GM05}  using a uniform pruning  of the branches
of  a random  recursive tree,  and \cite{S03}  using limit  of ancestral
processes obtained from super-critical Galton-Watson processes; see also
references therein for other related results.
\medskip 

Let   us  consider   the   ancestral  tree   $\ca(0)$  from   Definition
\ref{def:ancestral}  associated with  the stable  Lévy forest  under $\bP$
(that  is  for  the  stationary regime).  Let  $T>0$.  Conditionally  on
$\{M_{-T}^0=n-1\}$, that  is the  number of  individuals of  $\ca(0)$ at
level $-T$ is $n$, we label all  the $n$ individuals from $1$ to $n$
uniformly at random.  Define a continuous time ${\mathcal P}_{n}$-valued
process $(\Pi^{T,  [n]}(t), t\geq T)$, where $\Pi^{T,  [n]}(t)$ is the
partition of $[n]$  such that $i$ and  $j$ are in the same  block if and
only if  the $i$-th and $j$-th  individuals at level $-T$  have the same
ancestor at level $-t$ of the ancestral tree $\ca(0)$. By construction,
as $\lim_{t\rightarrow+\infty } M^0_{-t}=0$, 
we have that a.s. $\lim_{t\rightarrow +\infty } \Pi^{T,  [n]}(t)=[n]$.

\begin{prop}
\label{prop:BS}
Consider the sub-critical stable branching mechanism
  with immigration \eqref{eq:def-psi-b-s}. The law of $(\Pi^{T, [n]}(T\expp{\gamma
  t}), t\geq0)$, under $\bP (\cdot\, |M_{-T}^0=n-1)$, converges in the
sense of finite dimensional distribution to a Bolthausen-Sznitman
coalescent $\Pi^{[n]}$, as $b$ decreases to 1. 
\end{prop}
\begin{proof}

  Let $(\chi_t,t\ge  0)$ be  the GW  process with  branching rate  1 and
  offspring distribution  with generating  function $g_B$  introduced in
  Remark \ref{rem:size-biased}.  It is  well-known that,  conditioned on
  $\{\chi_T=n\}$, we obtain a  Markov coalescent process associated with
  the  genealogical tree  of  $\chi$ by  time-reversal.  But, by  Remark
  \ref{rem:size-biased} and Theorem \ref{theo:main-b-d}, we get that the
  process  $(\chi_t,0\le t\le  T)$  conditionally  on $\{\chi_T=n\}$  is
  distributed as the process  $(\tilde M_t+1,0\le t\le T)$ conditionally
  on $\{\tilde M_T+1=n\}$. Thus, the latter is a Markov process. Similar
  arguments   on  the   genealogical   tree  imply   that  the   process
  $(\Pi^{T,   [n]}(T\expp{\alpha   t}),    t\geq0)$   is   Markov   (but
  inhomogeneous in time).

  Then it is sufficient to show that
the transition rates of $(\Pi^{T, [n]}(T\expp{\gamma t}), t\geq0)$
converge to those of $\Pi^{[n]}$, as $b\rightarrow1+$. We
also notice that 
\[
(\#\Pi^{T, [n]}(T\expp{\gamma t})-1,
t\geq0)=(M_{-T\expp{\gamma t}}^0, t\geq0)
\]
and that the generations do not
overlap. Thus if $\#\Pi^{T, [n]}(T\expp{\gamma t})=k$, then any $m$ of the present 
blocks coalesce at rate 
\[
\frac{T\gamma \expp{\gamma t}}{\binom{k}{m}}{q^\mathrm{d}_{k-1, k-m}(T\expp{\gamma t})}
\]
where the death rates
$q^\mathrm{d}_{n,m}(t)=\lim_{\varepsilon\rightarrow0+} \varepsilon^{-1} \bP(M^0_{-(t+\varepsilon)}=m|
  M^0_{-t}=n) $
are computed in Section \ref{sec:d-process} for
general branching mechanism. 
Using Lemma \ref{lem:ratecov}, we deduce that for $2\leq m\leq k\leq n$:
\[
\lim_{b \rightarrow 1+} \frac{T\gamma \expp{\gamma
    t}}{\binom{k}{m}}{q^\mathrm{d}_{k-1, k-m}(T\expp{\gamma t})} =
\frac{(m-2)!(k-m)!}{(k-1)!}\cdot
\]
This proves the result. 
\end{proof}

\section{Critical stable case}\label{sec:cstable}

In this section only, we shall consider the critical stable case with branching
mechanism $\psi$ and immigration $\phi$ given by:
\begin{equation}
   \label{eq:def-psi-b-c}
\psi(\lambda)=\gamma\lambda^b  \quad\text{and}\quad
\phi(\lambda)=b\gamma\lambda^{b-1},
\end{equation}
with $\gamma>0$ and $b\in (1, 2]$. We also have (see Example 3.1 p.~62 in \cite{l:mvbmp}) for $\lambda\geq 0$ and
$t>0$:
\begin{equation}
   \label{eq:crit}
u(\lambda, t)=\frac{\lambda}{\left(1+\gamma
    (b-1)\lambda^{b-1}t\right)^{1/(b-1)}} 
\quad\text{and}\quad
c(t)=(\gamma (b-1) t)^{-\frac{1}{b-1}}.
\end{equation}

In this setting, both $M^0_{-t}$ and
$Z_0$  are infinite. For this reason, we only consider the families migrating to the system after some time $-T$.
\medskip

Let $T>0$.  Recall $\pi$ is  the Lévy measure in \reff{eq:def-psi}, $\N$
is the  corresponding excursion measure  on $\T$  of the Lévy  tree, and
$\P_r(d\bar \bff)$ is the probability distribution on $\F$ of the random
forest  $\cf=(\ct_i)_{i\in I}$  given by  the atoms  of a  poisson point
measure  on  $\T$  with  intensity $r\N(d\bt)$.   Similarly  to  Section
\ref{sec:random-forest}, we consider under $\bP$ a random leveled forest
$\bar\cf^{(T)}=(h_i,\cf_i)_{i\in  I^{(T)}}$  given  by the  atoms  of  a
Poisson point measure on $[-T, 0]\times\F$ with intensity
\[
\nu(dh, d\bff)= \ind_{[-T, +\infty )}(h)\, dh \left(\beta_1\N[d\bff]+\int_0^{+\infty}\pi(dr)\,
  \P_r(d\bff)\right), 
\]
and  let  $\bar  \ct^{(T)}=\bt(\bar   \cf^{(T)})$  be  the  random  tree
associated with this leveled forest.
Set for $a>-T$: 
\[
\ell^{a}(\bar  \ct^{(T)}) =\sum_{i\in
  I^{(T)}}\ell^{a-h_i}(\cf_i)\ind_{\{h_i\leq a\}} 
\quad\text{and}\quad  Z^{(T)}_a=\langle \ell^{a}(\bar
\ct^{(T)}),1\rangle. 
\]
Thanks to  the properties of
Poisson point  measures, we have, for $\lambda\geq 0$, $t\in [-T, +\infty )$:
\begin{align}\label{LapZT}
\bE\left[\expp{-\lambda Z^{(T)}_t}\right]&
=\exp\left\{-\int_{-T}^t\gamma b\,  u(\lambda, t-s)^{b-1}ds\right\}
\cr&
=\exp\left\{-\int_{-T}^t \frac{\gamma b\lambda^{b-1}}{1+\gamma
     (b-1)\lambda^{b-1}(t-s) }ds\right\}
\cr&
=\left(1+\gamma(b-1)\lambda^{b-1}(t+T)\right)^{-\frac{b}{b-1}}.
\end{align}
We write $ \ca^{(T)}(0)$ for the genealogical tree  $\ca_{\bar
  \ct^{(T)}}(0)$ of the extant
population. We define the ancestral process
$M^{(T)}= (M^{(T)}_t,t\in (-T, 0))$, where $1+M^{(T)}_t$ is  the number of ancestor
 of the 
extant population living at time $t$ by:
\[
M_t^{(T)}=\Card\{u\in\ca^{(T)}(0),\ H(u)=t\}-1.
\]

\begin{theo}\label{thm:gwiT}
  Consider  the critical  stable  branching  mechanism with  immigration
  \eqref{eq:def-psi-b-c}.   Then   the  time-changed   ancestral   process
  $(M^{(T)}_{-T\expp  {-t}},t\ge  0)$ is  distributed  under  $\bP $  as
 the GWI process  $(X_t,t\ge 0)$ from Theorem \ref{thm:gwi}.
\end{theo}

\begin{rem}
In Proposition 19 in \cite{BBS07}, it is shown that a reduced tree of a critical stable tree, after a deterministic time-changed, is a continuous-time Galton-Watson tree with birth rate $1$ and offspring distribution given by $g_B(r)$. We get here a similar result with an additional immigration mechanism.
\end{rem}
 
\begin{proof} 
We only give an outline of the proof as we follow the ideas of the proof
of Theorem \ref{thm:gwi}. In the critical case the function $c$ is given
by \reff{eq:crit} and the  function $g_t$ of \reff{eq:def-g} is still
given by formula  \reff{eq:g_stable}. Similarly to \reff{tau0},
\reff{taun}, we  define the jumping times $\{\tau_n^{(T)},\, n\geq0\}$
and jumpings sizes $\{\xi_n^{(T)},\, n\geq0\}$ of the ancestral process
$M^{(T)}$. We also define $\tau_n^{(T, B)}$ and $\tau_n^{(T, I)}$ as in
\reff{tauIB} with obvious change.  Note that in this setting,
$\tau_0^{(T)}$ is the immigration time of the first family after $-T$
which  survives up to  time $0$. So, we have for $t\in (0, T)$:
\[
\bP\left(\tau_0^{(T)}<t\right)
=\exp\left\{-\int_{-T}^{-t}ds\, \int_{(0,+\infty)}r\pi(dr)\P_r
  (H(\ct)>s)\right\} 
=\left(\frac{t}{T}\right)^{b/(b-1)}.
\]
Recall $g_I$ defined in \reff{eq:def-gi}. 
We also have, see \reff{eq:xi_0}, that for $t\in (0, T)$ and $r\in (0,
1)$: 
\[
\bP\left[r^{\xi_0^{(T)}}\bigm| \tau_0^{(T)}=t\right]=g_I(r). 
\]
Following the proof of Lemma \ref{lem:dist_Delta}, we get for $T>t>u>0$:
\begin{align*}
\bP\left(\tau_1^{(T, I)}<u\bigm|\tau_0^{(T)}
=t,  M^{(T)}_{-\tau_0^{(T)}}=k\right)
&=\left(\frac{u}{t}\right)^{\frac{b}{b-1}},\\
\bP\left(\tau_1^{(T, B)}<u\bigm|\tau_0^{(T)}
=t, M^{(T)}_{-\tau_0^{(T)}}=k\right)
&=\left(\frac{u}{t}\right)^k.
\end{align*}
This further implies that
\[
\bP\left(\tau_1^{(T)}<u\bigm|\tau_0^{(T)}=t,
  M^{(T)}_{-\tau_0^{(T)}}=k\right)
=\left(\frac{u}{t}\right)^{k+\frac{b}{b-1}}.
\]
We deduce that for $u>t>0$:
\[
\bP\left({\tau_1^{(T)}}<T\expp{-u}\bigm|{\tau_0^{(T)}}=T\expp{-t}, M^{(T)}_{{-\tau_0^{(T)}}}=n\right)=\expp{-\left(n+\frac{b}{b-1}\right)(u-t)}.
\]
Arguing as  in the proof  of Lemma \ref{lem:dist_jumps}, we  obtain that
given    $\tau_0^{(T)}$ and $\xi_0^{(T)}=n$,  $ \tau_1^{(T)}$ and
$\xi_1^{(T)} $ are independent  and the
conditional     generating    function     of    $\xi_1^{(T)}$  is
given by $g_{[n]}$ defined in \reff{eq:gf_jumps2}.  The end of the proof
is then similar. 
\end{proof}

The following proposition, whose proof is left to the reader, is
parallel to Corollary 6.5 in \cite{CD12}. 
\begin{prop}
\label{CormartT}
Consider     the    critical     stable    branching     mechanism
  with immigration   \eqref{eq:def-psi-b-c}. 
Then, we have:
\[
\lim_{t\rightarrow0+}\frac{M_{-t}^{(T)}}{c(t)}\overset{a.s.}{=}Z_0^{(T)}
\]
\end{prop}
We order the set $\{i\in I^{(T)},\ h_i<0\mbox{ and
}\ell^{-h_i}(\cf_i)\ne 0\}$ of the  immigrants that have descendants at
time 0, 
by the date of arrival of the immigrant: $I_0^{(T)}=\{i_k, k\ge 0\}$ 
with $-\tau_0^{(T)}=h_{i_0}<h_{i_1}<h_{i_2}<\cdots<0$.
For every $k\ge 0$, we set $\zeta_k^{(T)}$ the size of the population at time
0 generated by the $k$-th immigrant, that is
$ \zeta_k^{(T)}=\langle
\ell^{-k_{i_k}}(\cf_{i_k}),1\rangle$.
Notice that $\sum_{k=0}^{+\infty}\zeta_k^{(T)}=Z_0^{(T)}$.
With Theorem \ref{thm:gwiT} and Proposition \ref{CormartT} in hand, the
next two results follow by the same arguments as Proposition
\ref{PRMage} and Corollary \ref{propPD}, respectively. 

\begin{prop}\label{PRMageT}
Consider     the    critical     stable    branching     mechanism
    with immigration \eqref{eq:def-psi-b-c}. The random point measure $\sum_{k\in
    \N}\delta_{c(T) \zeta_k^{(T)}}(dx)$ is 
 a Poisson point  measure on $[0, \infty)$ with intensity $
 g(x)\, dx$, with $g$ defined by \reff{Fung}.
\end{prop}

\begin{proof}
  Recall $W$ from Corollary \ref{cor:martlim}. According to \reff{LapZT}
  and \reff{eq:LaplaceW},  we deduce  that $c(T)Z_0^{(T)}$ and  $W$ have
  the  same distribution.  Recall  $X^i$  and $W_i$  from  the proof  of
  Proposition  \ref{PRMage}.  Arguing as in the proof of Proposition
  \ref{PRMage}, and using  Theorem  \ref{thm:gwiT}  and
  Proposition \ref{CormartT}, we get that:
\[
\left( c(T)\zeta_i^{(T)}, 
  i\in \N\right)\overset{d}{=}\left(\expp{-\frac{T_i}{(b-1)}}W_i:
  i\in \N\right).
\]
Then use  \eqref{eqnlaw} and Proposition \ref{PRMage} to conclude.
\end{proof}

Using  Proposition  \ref{PRMageT},  we  obtain  directly  the  following
results,  which  is  the  analogue  to  Corollary   \ref{propPD}. 

\begin{cor}\label{propPDT} 
Consider the critical quadratic  branching mechanism   with immigration 
\eqref{eq:def-psi-b-c} with $b=2$. 
  Let $(\zeta_{(k)}^{(T)}, k\in \N)$  be  the
decreasing  order statistics of $(\zeta_k^{(T)}, k\in \N)$. Then, the
random sequence $\left( \zeta_{(k)}^{(T)}/Z_0^{(T)} , k\in \N\right)
$ has a Poisson-Dirichlet distribution with parameter $2$.
\end{cor}
\medskip

One can also consider the critical CBI associated as the limit of the
sub-critical CBI when $\alpha$ in \reff{eq:def-psi} goes down to
$0$. For the stable case,  consider the  birth rates $q^\mathrm{b}_{
  n,m}(-t)$ defined in Remark
\ref{rem:birth-rate-b} for $\psi(\lz)=\alpha\lambda+\gamma\lambda^\mathrm{b}$ with $b\in
(1, 2]$,  $\gamma>0$ and $\alpha>0$. Letting $\alpha$ goes down to 0,
wet get 
$ \lim_{\alpha\rightarrow0}q^\mathrm{b}_{
  n,m}(-t)=q^{\mathrm{b},0}_{n,m}(-t)$  with:
\begin{equation}
   \label{eq:b-crit}
q^{\mathrm{b},0}_{n,m}(-t)=
\begin{cases}\frac{(m+1)|
    b(b-2)\cdots(b-m+n)|}{(m-n+1)!}\frac{1}{t}\quad &\text{ for $b\in
    (1, 2)$ and $m>n$},\\
\frac{n+2}{t}\quad &\text{ for $b=2$ and $m=n+1$},\\
0\quad &\text{ for $b=2$ and $ m>n+1$}.
\end{cases}
\end{equation}

Then, Theorem \ref{thm:gwiT} and \reff{rateX} implies the following proposition that shows that both constructions for the critical case coincide.

\begin{prop}
Assume $\psi(\lz)=\gamma\lambda^b$ with $b\in (1, 2]$. Then, the
ancestral process
$(M_{t}^{(T)}, \,  -T\leq t<0)$ is a Markov chain with birth rate
$q^{\mathrm{b},0}_{n,m}(-t)$ given by \reff{eq:b-crit}
for $m>n\geq 0$ and $t>0$. 
\end{prop}

\bibliographystyle{abbrv}
\bibliography{biblio}

\end{document}